\documentclass[11pt,reqno]{amsart}

\title{Physical measures of discretizations of generic diffeomorphisms}
\author{Pierre-Antoine Guihéneuf}
\address{Universit\'e Paris-Sud \\ Universidade Federal Fluminense, Niterói}
\email{pierre-antoine.guiheneuf@math.u-psud.fr}

\usepackage[latin1]{inputenc}
\usepackage[english,french]{babel}
\usepackage[T1]{fontenc}
\usepackage{amsfonts}
\usepackage{amsmath}
\usepackage{amssymb}
\usepackage{stmaryrd}
\usepackage{amsthm}
\usepackage{enumerate}
\usepackage{pgf,tikz}
\usetikzlibrary{decorations.pathreplacing, shapes.multipart, arrows, matrix}
\usepackage[margin=0pt]{caption}
\usepackage[pdftex,colorlinks=true,linkcolor=blue,citecolor=blue,urlcolor=blue]{hyperref}

\newtheorem{lemme}{Lemma}
\newtheorem{theoreme}[lemme]{Theorem}
\newtheorem{prop}[lemme]{Proposition}
\newtheorem{coro}[lemme]{Corollary}

\newtheorem{add}[lemme]{Addendum}

\newtheorem{theo}{Theorem}

\theoremstyle{definition}
\newtheorem{definition}[lemme]{Definition}

\theoremstyle{remark}
\newtheorem{rem}[lemme]{Remark}

\newcommand{\N}{\mathbf{N}}

\newcommand{\R}{\mathbf{R}}

\newcommand{\T}{\mathbf{T}}

\newcommand{\Z}{\mathbf{Z}}
\newcommand{\varep}{\varepsilon}
\newcommand{\Hom}{\operatorname{Homeo}}

\newcommand{\Leb}{\operatorname{Leb}}
\newcommand{\Diff}{\operatorname{Diff}}

\newcommand{\card}{\operatorname{Card}}
\newcommand{\dist}{\operatorname{dist}}

\newcommand{\1}{\mathbf 1}

\newcommand{\Id}{\operatorname{Id}}
\newcommand{\im}{\operatorname{im}}

\newcommand{\Sp}{\mathbf{S}}

\makeatletter
\newenvironment{abstracts}{%
  \ifx\maketitle\relax
    \ClassWarning{\@classname}{Abstract should precede
      \protect\maketitle\space in AMS document classes; reported}%
  \fi
  \global\setbox\abstractbox=\vtop \bgroup
    \normalfont\Small
    \list{}{\labelwidth\z@
      \leftmargin3pc \rightmargin\leftmargin
      \listparindent\normalparindent \itemindent\z@
      \parsep\z@ \@plus\p@
      
      \itemsep\bigskipamount
    }%
}{%
  \endlist\egroup
  \ifx\@setabstract\relax \@setabstracta \fi
}

\newcommand{\abstractin}[1]{%
  \otherlanguage{#1}%
  \item[\hskip\labelsep\scshape\abstractname.]%
}
\makeatother

\begin{document}

\begin{abstracts}
\abstractin{english}
What is the ergodic behaviour of numerically computed segments of orbits of a diffeomorphism? In this paper, we try to answer this question for a generic conservative $C^1$-diffeomorphism, and segments of orbits of Baire-generic points. The numerical truncation will be modelled by a spatial discretization. Our main result states that the uniform measures on the computed segments of orbits, starting from a generic point, accumulates on the whole set of measures that are invariant under the diffeomorphism. In particular, unlike what  could be expected naively, such numerical experiments do not see the physical measures (or more precisely, cannot distinguish physical measures from the other invariant measures).

\abstractin{french}
Que se passe-t-il d'un point de vue ergodique lorsqu'on calcule de longs segments d'orbite d'un difféomorphisme, à précision numérique fixée ? On tente ici de répondre à cette question pour un $C^1$-difféomorphisme conservatif générique, et pour la plupart des orbites au sens de Baire. L'opération de troncation numérique sera modélisée par une discrétisation spatiale. Notre principal résultat exprime que les mesures uniformes sur les segments d'orbites calculés, en partant d'un point générique, s'accumulent sur l'ensemble des mesures invariantes par le difféomorphisme. En particulier, de telles simulations numériques ne permettent pas de distinguer les mesures physiques du système parmi toutes les autres mesures invariantes.
\end{abstracts}

\selectlanguage{english}

\maketitle

\setcounter{tocdepth}{1}
\tableofcontents

\section{Introduction}

This paper is devoted to the study of the physical measures of the discretizations of a generic conservative diffeomorphism. Recall the classical definition of a physical measure for a map $f$.

\begin{definition}\label{sport}
Let $X$ be a compact manifold equipped with a Lebesgue measure $\Leb$, and $f : X\to X$. A Borel probability measure $\mu$ is called \emph{physical} (also called Sinai-Ruelle-Bowen) for the map $f$ if its basin of attraction has positive Lebesgue measure, where the \emph{basin of attraction} of $\mu$ for $f$ is the set
\[\left\{x\in X\ \big\vert\ \frac{1}{M} \sum_{m=0}^{M-1} \delta_{f^m(x)} \underset{M\to+\infty}{\longrightarrow} \mu\right\}\]
of points whose Birkhoff's limit coincides with $\mu$ (where the convergence of measures is taken in the sense of weak-* topology).
\end{definition}

Heuristically, the physical measures are the ones that can be observed in practice, because they are ``seen'' by a ``large ''set of points $x$.

The question of existence, stochastic stability, dependence with respect to parameters, etc. have been extensively studied (see for example the quite old surveys \cite{MR1933431} and \cite{VianaStoch}). Here, our aim is to study similar concepts in the view of discretizations: which measures can be seen by the discretizations of generic conservative $C^1$-diffeomorphisms?
\bigskip

In this paper we will consider that the space phase is the torus $\T^n$ endowed with Lebesgue measure. We will model the numerical truncation made by a computer by a spatial discretization. Thus, we define the uniform grids
\[E_N = \left\{ \left(\frac{i_1}{N},\cdots,\frac{i_n}{N}\right)\in \R^n/\Z^n \middle\vert\ 1\le i_1,\cdots,i_n \le N\right\}.\]
In particular, if $N=10^k$, then $E_N$ represents the set of points whose coordinates are decimal numbers with a most $k$ decimal places. We then take $P_N : \T^n\to E_N$ a projection on the nearest point of $E_N$; in other words $P_N(x)$ is (one of) the point(s) of $E_N$ which is the closest from $x$. This allows to define the discretizations of $f$.

\begin{definition}
The \emph{discretization} $f_N : E_N\to E_N$ of $f$ on the grid $E_N$ is the map $f_N = P_N\circ f_{|E_N}$.
\end{definition}

In particular, if $N=10^k$, then $f_N$ models the map which is iterated by the computer when it works with $k$ digits.

We will see in Section~\ref{AddendSett} that the quite restrictive framework of the torus $\T^n$ equipped with the uniform grids can be generalized to arbitrary manifolds, provided that the discretizations grids behave locally (and almost everywhere) like the canonical grids on the torus.

We denote by $\mu_{x}^{f_N}$ the limit of the Birkhoff sums
\[\frac{1}{M} \sum_{m=0}^{M-1} \delta_{f_N^m(x_N)}.\]
More concretely, $\mu_{x}^{f_N}$ is the $f_N$-invariant probability measure supported by the periodic orbit on which the positive orbit of $x_N=P_N(x)$ falls after a while. We would like to know the answer the following question: for a generic conservative $C^1$-diffeomorphism $f$, does the sequence of measure $\mu_{x}^{f_N}$ tend to a physical measure of $f$ for most of the points $x$ as $N$ goes to infinity?

The corresponding $C^0$ case has already been treated in \cite{Guih-discr}:

\begin{theoreme}[Guihéneuf]\label{mesinv}
For a generic homeomorphism $f\in\Hom(\T^n,\Leb)$, for any $f$-invariant probability measure $\mu$, there exists a subsequence $(N_k)_k$ of discretizations such that for any point $x\in \T^n$,
\[\mu_{x}^{f_{N_k}} \underset{k\to+\infty}{\longrightarrow} \mu.\]
\end{theoreme}

It implies in particular that for a generic homeomorphism $f\in\Hom(\T^n,\Leb)$ and every $x\in\T^n$, the measures $\mu^{f_N}_x$ accumulate on the whole set of $f$-invariant measures when $N$ goes to infinity (moreover, given an $f$-invariant measure $\mu$, the sequence $(N_k)_{k\ge 0}$ such that $\mu^{f_{N_k}}_x$ tends to $\mu$ can be chosen independently of $x$). In a certain sense, this theorem in the case of homeomorphisms expresses that from the point of view of the discretizations, all the $f$-invariant measures are physical.

In the $C^1$-case, it can be easily obtained that for a generic conservative $C^1$-diffeomorphism $f$, any $f$-invariant measure is the limit of a sequence of $f_N$-invariant measures (Corollary~10.9 of \cite{Guih-These}). This is a consequence of an ergodic closing lemma of R.~Mañé and F.~Abdenur, C.~Bonatti and S.~Crovisier (see \cite{MR2811152}); however it does not say anything about the basin of attraction of these discrete measures.

In this paper, we improve this statement for generic conservative $C^1$-diffeomorphisms, in order to describe the basin of attraction of the discrete measures. In particular, we prove the following result (Theorem~\ref{TheoMesPhysDiff}).

\begin{theo}\label{TheoA}
For a generic diffeomorphism $f\in\Diff^1(\T^n,\Leb)$, for a generic point $x\in \T^n$, for any $f$-invariant probability measure $\mu$, there exists a subsequence $(N_k)_k$ of discretizations such that
\[\mu_{x}^{f_{N_k}} \underset{k\to+\infty}{\longrightarrow} \mu.\]
\end{theo}

Notice that given an $f$-invariant measure $\mu$, the sequence $(N_k)_k$ such that $\mu_{x}^{f_{N_k}}$ converges to $\mu$ depends on the point $x$, contrary to what happens in the $C^0$ case.

Remark that in Theorem~\ref{TheoA}, the generic set of points $x$ depends on the diffeomorphism. However, we will also prove that if we fix a countable subset $D\subset \T^n$, then for a generic conservative $C^1$-diffeomorphism $f$ and for any $x\in D$, the measures $\mu^{f_N}_x$ accumulate on the whole set of $f$-invariant measures (Addendum~\ref{AddTheoMesPhysDiff}). This is a process that is usually applied in practice to detect the $f$-invariant measures: fix a finite set $D\subset \T^n$ and compute the measure $\mu^{f_N}_x$ for $x\in D$ and for a large order of discretization $N$. Our theorem expresses that it is possible that the measure that we observe on numerical experiments is very far away from the physical measure.

\label{AvilaErgo}Note that in the space $\Diff^1(\T^n,\Leb)$, there are open sets where generic diffeomorphisms are ergodic: the set of Anosov diffeomorphisms is open in $\Diff^1(\T^n,\Leb)$, and a generic Anosov conservative $C^1$-diffeomorphism is ergodic (it is a consequence of the fact that any $C^2$ Anosov conservative diffeomorphism is ergodic (see for instance \cite{MR0224771}), together with the theorem of regularization of conservative diffeomorphisms of A.~Avila \cite{MR2736152}). More generally, A.~Avila, S.~Crovisier and A.~Wilkinson have set recently in \cite{ArturSylvain} a generic dichotomy for a conservative diffeomorphism $f$: either $f$ is ergodic, either all the Lyapunov exponents of $f$ vanish. In short, there are open sets where generic conservative diffeomorphisms have only one physical measure; in this case, our result asserts that this physical measure is not detected on discretizations by computing the measures $\mu^{f_N}_x$.

Recall that results of stochastic stability are known to be true in various contexts (for example, expanding maps \cite{MR884892},\cite{MR874047}, \cite{MR685377}, uniformly hyperbolic attractors \cite{MR874047}, \cite{MR857204}, etc.). These theorems suggest that the physical measures can always be observed in practice, even if the system is noisy. Theorem~\ref{TheoA} indicates that the effects of discretizations (i.e. numerical truncation) might be quite different from those of a random noise.

However, we shall remark that the arguments of the beginning of the proof of Theorem~\ref{TheoA} implies that for a generic diffeomorphism $f\in\Diff^1(\T^n,\Leb)$ and a generic point $x\in\T^n$ (or equivalently, for any $x\in\T^n$ and for a generic $f\in\Diff^1(\T^n,\Leb)$), the measures
\[\mu^f_{x,m} = \frac 1m \sum_{i=0}^{m-1}{f}_*^i \delta_x\]
accumulate on the whole set of $f$-invariant measures. Thus, it would be nice to obtain a statement similar to Theorem~\ref{TheoA}, but where the hypothesis ``for a Baire-generic set of points $x$'' would be replaced by ``for Lebesgue almost all point $x$''.
\bigskip

Note that Theorem~\ref{TheoA} does not say anything about the measures $\mu_{\T^n}^{f_N}$, defined as follows. Consider $\Leb_N$ the uniform measure on the grid $E_N$ and set
\[\mu_{\T^n}^{f_N} = \lim_{M\to+\infty} \frac{1}{M} \sum_{m=0}^{M-1} (f_N^m)_*(\Leb_N).\]
The limit in the previous equation is well defined: the measure $\mu_{\T^n}^{f_N}$ is supported by the union of periodic orbits of $f_N$, and the total measure of each of these periodic orbits is proportional to the size of its basin of attraction under $f_N$. For now, the theoretical study of the measures $\mu_{\T^n}^{f_N}$ for generic conservative $C^1$-diffeomorphisms seems quite hard, as this kind of questions is closely related to the still open problem of genericity of ergodicity among these maps (see \cite{ArturSylvain} for the most recent advances on this topic).

On numerical simulations of these measures $\mu_{\T^n}^{f_N}$, it is not clear whether they converge towards Lebesgue measure or not (see Figures~\ref{MesC1IdCons2p}, \ref{MesC1AnoCons2p} and \ref{MesC1AnoConsSerie}). However, one can hope that their behaviour is not as erratic as for generic conservative homeomorphisms, where they accumulate on the whole set of $f$-invariant measures (see Section~4.6 of \cite{Guih-discr}). Indeed, \cite{Gui15a} shows that a simple dynamical invariant associated to each discretization $f_N$ (the degree of recurrence) converges to 0 as $N$ tends to $+\infty$ for a generic conservative $C^1$-diffeomorphism, while it accumulates on the whole segment $[0,1]$ for generic conservative homeomorphisms; this shows that the discretizations of generic conservative $C^1$-diffeomorphism and homeomorphisms might have quite different dynamical characteristics.
\bigskip

At the end of this paper, we also present numerical experiments simulating the measures $\mu_x^{f_N}$ for some examples of conservative $C^1$-diffeomorphisms $f$ of the torus. The results of these simulations are quite striking for an example of $f$ $C^1$-close to $\Id$ (see Figure~\ref{MesPhysIdC1}): even for very large orders $N$, the measures $\mu_x^{f_N}$ do not converge to Lebesgue measure at all, and depend quite dramatically on the integer $N$. This illustrates perfectly Theorem~\ref{TheoMesPhysDiff} (more precisely, Addendum~\ref{AddTheoMesPhysDiff}), which states that if $x$ is fixed, then for a generic $f\in\Diff^1(\T^2,\Leb)$, the measures $\mu_x^{f_N}$ accumulate on the whole set of $f$-invariant measures, but do not say anything about, for instance, the frequency of orders $N$ such that $\mu_x^{f_N}$ is not close to Lebesgue measure. Moreover, the same phenomenon (although less pronounced) occurs for diffeomorphisms close to a linear Anosov automorphism (Figure~\ref{MesPhysAnoC1}).
\bigskip

The proof of Theorem~\ref{TheoA} uses crucially results about the dynamics of discretizations of generic sequences of linear maps: the discretization of a linear map $A\in SL_n(\R)$ is a map $\widehat A : \Z^n\to Z^n$, such that $\widehat A(x)$ is the point of $\Z^n$ which is the closest of $Ax$. In particular, Lemma~\ref{DerTheoPart2} (whose proof is in appendix, because it is already done it \cite{Gui15a}) expresses that the preimage of some points of $\Z^n$ by a generic sequence of discretizations have a big cardinality. This will allow us to merge some orbits of the discretizations. The proof of Theorem~\ref{TheoA} also uses two connecting lemmas (the connecting lemma for pseudo-orbits of \cite{MR2090361} and an improvement of the ergodic closing lemma of \cite{MR2811152}) and a statement of local linearization of a $C^1$-diffeomorphism (Lemma~\ref{LemExtension}) involving the regularization result due to A.~Avila \cite{MR2736152}.
\bigskip

\subsection{Acknowledgements}

Je remercie très chaleureusement Sylvain Crovisier pour son aide précieuse concernant les lemmes de perturbation, ainsi que François Béguin à qui cet article doit beaucoup.

\section{Statement of the theorem and sketch of proof}\label{Section!}

We recall the statement of the main theorem of this paper (stated as Theorem~\ref{TheoA} in the introduction).

\begin{theoreme}\label{TheoMesPhysDiff}
For a generic diffeomorphism $f\in\Diff^1(\T^n,\Leb)$, for a generic point $x\in \T^n$, for any $f$-invariant probability measure $\mu$, there exists a subsequence $(N_k)_k$ of discretizations such that
\[\mu_{x}^{f_{N_k}} \underset{k\to+\infty}{\longrightarrow} \mu.\]
\end{theoreme}

Remark that the theorem in the $C^0$ case is almost the same, except that here, the starting point $x\in\T^n$ is no longer arbitrary but has to be chosen in a generic subset of the torus, and that the sequence $(N_k)_k$ depends on the starting point $x$. The proof of this theorem will also lead to the two following statements.

\begin{add}\label{AddTheoMesPhysDiff}
For a generic diffeomorphism $f\in\Diff^1(\T^n,\Leb)$, for any $\varep>0$ there exists a $\varep$-dense subset $(x_1,\cdots,x_m)$ such that for any $f$-invariant probability measure $\mu$, there exists a subsequence $(N_k)_k$ of discretizations such that for every $j$,
\[\mu_{x_j}^{f_{N_k}} \underset{k\to+\infty}{\longrightarrow} \mu.\]

Also, for any countable subset $D\subset \T^n$, for a generic diffeomorphism $f\in\Diff^1(\T^n,\Leb)$, for any $f$-invariant probability measure $\mu$, and for any finite subset $E\subset D$, there exists a subsequence $(N_k)_k$ of discretizations such that for every $x\in E$, we have
\[\mu_{x}^{f_{N_k}} \underset{k\to+\infty}{\longrightarrow} \mu.\]
\end{add}

The first statement asserts that if $f$ is a generic conservative $C^1$-diffeomorphism, then for any $f$-invariant measure $\mu$, there exists an infinite number of discretizations $f_N$ which possess an invariant measure which is close tu $\mu$, and whose basin of attraction is $\varep$-dense. Basically, for an infinite number of $N$ any $f$-invariant will be seen from any region of the torus.

In the second statement, a countable set of starting points of the experiment is chosen ``by the user''. This is quite close to what happens in practice: we take a finite number of points $x_1,\cdots,x_m$ and compute the measures $\mu_{x_m,T}^{f_{N_k}}$ for all $m$, for a big $N\in\N$ and for ``large'' times $T$ (we can expect that $T$ is large enough to have $\mu_{x_m,T}^{f_{N_k}} \simeq \mu_{x_m}^{f_{N_k}}$). In this case, the result expresses that it may happen (in fact, for arbitrarily large $N$) that the measures $\mu_{x_m,T}^{f_{N_k}}$ are not close to the physical measure of $f$ but are rather chosen ``at random'' among the set of $f$-invariant measures.
\bigskip

We also have a dissipative counterpart of Theorem~\ref{TheoMesPhysDiff}, whose proof is easier.

\begin{theoreme}\label{TheoMesPhysDiffDissip}
For a generic dissipative diffeomorphism $f\in\Diff^1(\T^n)$, for any $f$-invariant probability measure $\mu$ such that the sum of the Lyapunov exponents of $\mu$ is negative (or equal to 0), for a generic point $x$ belonging to the same chain recurrent class as $\mu$, there exists a subsequence $(N_k)_k$ of discretizations such that
\[\mu_{x}^{f_{N_k}} \underset{k\to+\infty}{\longrightarrow} \mu.\]
\end{theoreme}

Remark that if we also consider the inverse $f^{-1}$ of a generic diffeomorphism $f\in\Diff^1(\T^n)$, we can recover any invariant measure $\mu$ of $f$ by looking at the measures $\mu_{x}^{f_{N_k}}$ for generic points $x$ in the chain recurrent class of $\mu$.

The proof of this result is obtained by applying Lemma~\ref{RemplaceDissip} during the proof of Theorem~\ref{TheoMesPhysDiff}.
\bigskip

We also have the same statement as Theorem~\ref{TheoMesPhysDiff} but for expanding maps of the circle. We denote $\mathcal E_d^1(\Sp^1)$ the set of $C^1$-expanding maps of the circle of degree $d$.

\begin{prop}\label{TheoMesPhysDiffExp}
For a generic expanding map $f\in\mathcal E_d^1(\Sp^1)$, for any $f$-invariant probability measure $\mu$, for a generic point $x\in\Sp^1$, there exists a subsequence $(N_k)_k$ of discretizations such that
\[\mu_{x}^{f_{N_k}} \underset{k\to+\infty}{\longrightarrow} \mu.\]
\end{prop}

The proof of this statement is far easier than that of Theorem~\ref{TheoMesPhysDiff} , as it can be obtained by coding any expanding map of class $C^1$ (that is, any $f\in\mathcal E_d^1(\Sp^1)$ is homeomorphic to a full shift on a set with $d$ elements).
\bigskip

We will use the connecting lemma for pseudo-orbits (see \cite{MR2090361}), together with an ergodic closing lemma (adapted from \cite{MR2811152}) and the results of the appendix on the fact that the asymptotic rate is null (in particular Lemma~\ref{DerTheoPart2}), to prove that any invariant measure of the diffeomorphism can be observed by starting at any point of a generic subset of $\T^n$.

By Baire theorem and the fact that for a generic conservative diffeomorphism, a generic invariant measure is ergodic, non periodic and has no zero Lyapunov exponent (see Theorem~3.5 of \cite{MR2811152}), the proof of Theorem~\ref{TheoMesPhysDiff} can be reduced easily to that of the following approximation lemma.

\begin{lemme}\label{LemMesPhysDiff}
For every $f\in\Diff^1(\T^n,\Leb)$, for every $f$-invariant measure $\mu$ which is ergodic, not periodic and has no zero Lyapunov exponent, for every open subset $U\subset \T^n$, for every $C^1$-neighbourhood $\mathcal V$ of $f$, for every $\varep>0$ and every $N_0\in\N$, there exists $g\in\Diff^1(\T^n,\Leb)$ such that $g\in \mathcal V$, there exists $y\in U$ and $N\ge N_0$ such that $\dist(\mu, \mu^{g_N}_y)<\varep$. Moreover, we can suppose that this property remains true on a whole neighbourhood of $g$.
\end{lemme}

First of all, we explain how to deduce Theorem~\ref{TheoMesPhysDiff} from Lemma~\ref{LemMesPhysDiff}.

\begin{proof}[Proof of Theorem~\ref{TheoMesPhysDiff}]
We consider a sequence $(\nu_\ell)_{\ell\ge 0}$ of Borel probability measures, which is dense in the whole set of probability measures. We also consider a sequence $(U_i)_{i\ge 0}$ of open subsets of $\T^n$ which spans the topology of $\T^n$. This allows us to set
\[\mathcal S_{\,N_0,k_0,\ell,i} = \left\{ f\in\Diff^1(\T^n,\Leb)\ \middle\vert\ \begin{array}{l}\exists \mu\ f\text{-inv.} : \dist(\mu,\nu_\ell)\le 1/k_0\implies\\
\exists N\ge N_0, y\in U_i : d(\mu^{f_N}_y ,\nu_\ell)<2/k_0 \end{array}\right\}.\]

We easily see that the set
\[\bigcap_{N_0,k_0,\ell,i \ge 0} \mathcal S_{\,N_0,k_0,\ell,i}\]
in contained in the set of diffeomorphisms satisfying the conclusions of the theorem.

It remains to prove that each set $\mathcal S_{\,N_0,k_0,\ell,i}$ contains an open and dense subset of $\Diff^1(\T^n,\Leb)$. Actually the interior of each set $\mathcal S_{\,N_0,k_0,\ell,i}$ is dense. This follows from the upper semi-continuity of the set of $f$-invariant measures with respect to $f$ and from the combination of Lemma~\ref{LemMesPhysDiff} with the fact that for a generic diffeomorphism, a generic invariant measure is ergodic, non periodic and has no zero Lyapunov exponent (see Theorem~3.5 of \cite{MR2811152}).
\end{proof}

It remains to prove Lemma~\ref{LemMesPhysDiff}. We now outline the main arguments of this quite long and technical proof.

\paragraph{Sketch of proof of Lemma~\ref{LemMesPhysDiff}.}
First of all, we take a point $x\in\T^n$ which is typical for the measure $\mu$. In particular, by an ergodic closing lemma derived from that of F.~Abdenur, C.~Bonatti and S.~Crovisier \cite{MR2811152} (Lemma~\ref{ErgoLemPlus}), there is a perturbation of $f$ (still denoted by $f$) so that the orbit $\omega$ of $x$ is periodic of period $\tau_1$; moreover, $\omega$ can be supposed to bear an invariant measure close to $\mu$, to have an arbitrary large length, and to have Lyapunov exponents and Lyapunov subspaces close to that of $\mu$ under $f$. Applying the (difficult) connecting lemma for pseudo-orbits of C.~Bonatti and S.~Crovisier \cite{MR2090361}, we get another perturbation of the diffeomorphism (still denoted by $f$), such that the stable manifold of $x$ under $f$ meets the open set $U$ at a point that we denote by $y$.

So, we need to perturb the diffeomorphism $f$ so that:
\begin{itemize}
\item the periodic orbit $x$ is stabilized by $f_N$. This can be easily made by a small perturbation of $f$;
\item the positive orbit of $y$ under $f_N$ falls on the periodic orbit of $x$ under $f_N$. This is the difficult part of the proof: we can apply the previous strategy to put every point of the positive orbit of $y$ on the grid only during a finite time. It becomes impossible to perform perturbations to put the orbit of $y$ on the grid --- without perturbing the orbit of $x$ --- as soon as this orbit comes into a $C/N$- neighbourhood of the orbit of $x$ (where $C$ is a constant depending on $\mathcal V$).
\end{itemize}
To solve this problem, we need the results about the linear case we prove in the appendix. if $n$ is large enough, at the scale of the grid $E_N$, the diffeomorphism $f$ is linear. Thus, iterating the discretization of $f$is equivalent to iterate a discretization of the linear cocycle given by the differentials of $f$. We can thus apply the results of the linear case (Lemma~\ref{DerTheoPart2}), which allow us to merge the positive orbits of $x$ and $y$ under the discretization.
\bigskip

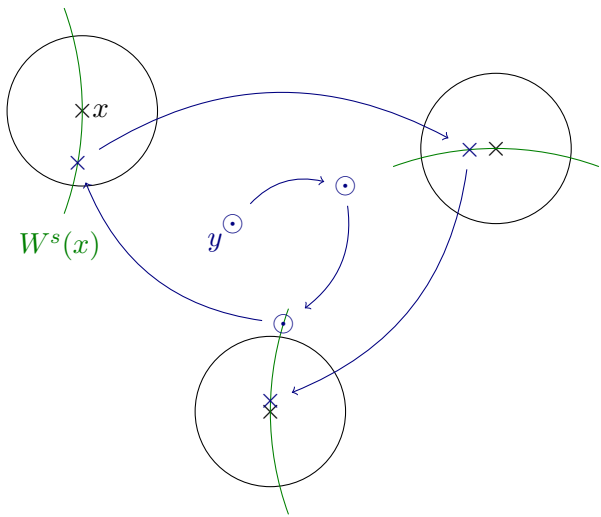
\begin{figure}[t]
\begin{center}
\begin{tikzpicture}[scale=1]

\draw (0,0) node{$\times$};
\draw (0,0) circle (1);
\draw[color=green!50!black] (0,0) arc (180:200:4);
\draw[color=green!50!black] (0,0) arc (180:160:4);

\draw (3,3.5) node{$\times$};
\draw (3,3.5) circle (1);
\draw[color=green!50!black] (3,3.5) arc (90:110:4);
\draw[color=green!50!black] (3,3.5) arc (90:70:4);

\draw (-2.5,4) node{$\times$};
\draw (-2.5,4) node[right]{$x$};
\draw (-2.5,4) circle (1);
\draw[color=green!50!black] (-2.5,4)  arc (0:20:4);
\draw[color=green!50!black] (-2.5,4)  arc (0:-20:4);
\draw[color=green!50!black] (-2.8,2.2)  node{$W^s(x)$};

\node[color=blue!50!black] (A) at (-.5,2.5) {$\odot$};
\node[color=blue!50!black] (B) at (1,3) {$\odot$};
\node[color=blue!50!black] (C) at (.175,1.17) {$\odot$};
\node[color=blue!50!black] (D) at (-2.561,3.31) {$\times$};
\node[color=blue!50!black] (E) at (2.651,3.485) {$\times$};
\node[color=blue!50!black] (F) at (0.002,0.140) {$\times$};

\draw[->,color=blue!50!black] (A) to[bend left] (B);
\draw[->,color=blue!50!black] (B) to[bend left] (C);
\draw[->,color=blue!50!black] (C) to[bend left] (D);
\draw[->,color=blue!50!black] (D) to[bend left] (E);
\draw[->,color=blue!50!black] (E) to[bend left] (F);

\draw[->,color=blue!50!black] (A) node[below left]{$y$};

\end{tikzpicture}
\caption[Perturbation of Lemma~\ref{LemMesPhysDiff}]{During the proof of Lemma~\ref{LemMesPhysDiff}, it is easy to perturb the first points of the orbit of $y$ (small disks) until the orbit meets the neighbourhoods of the orbits of $x$ where the diffeomorphism is linear (inside of the circles). The difficulty of the proof is to make appropriate perturbations in these small neighbourhoods.}\label{DessinGlobPhys}
\end{center}
\end{figure}

In more detail, we use Lemma~\ref{LemExtension} to linearize locally the diffeomorphism in the neighbourhood of the periodic orbit of $\omega$. In particular, the positive orbit of $y$ eventually belongs to this linearizing neighbourhood, from a time $T_1$. We denote $y' = f^{T_1}(y)$. To summarize, the periodic orbit $\omega$ bears a measure close to $\mu$, its Lyapunov exponents are close to that of $\mu$, its Lyapunov linear subspaces are close to that of $\mu$ (maybe not all along the periodic orbit, but at least for the first iterates of $x$). The diffeomorphism $f$ is linear around each point of $\omega$. Finally, the stable manifold of $\omega$ meets $U$ at $y$, and the positive orbit of $y$ is included in the neighbourhood of $\omega$ where $f$ is linear from the point $y' = f^{T_1}(y)$.

We then choose an integer $N$ large enough, and perturb the orbit of $x$ such that it is stabilized by the discretization $f_N$. We want to make another perturbation of $f$ such that the backward orbit of $x$ by $f_N$ also contains $y'$ (recall that $f_N$ is not necessarily one-to-one). This is done by a perturbation supported in the neighbourhood of
$\omega$ where $f$ is linear. First of all, during a time $t_4\ge 0$, we apply Lemma~\ref{DerTheoPart2} to find a point $z$ in the neighbourhood of $f^{-t_4}(x)$ where $f$ is linear, but far enough from $f^{-t_4}(x)$ compared to $1/N$, such that the $t_4$-th image of $z$ by the discretization $f_N$ is equal to $x$. Next, we perturb the orbit of $z$ under $f^{-1}$ during a time $t_3\ge 0$ such that $f^{-t_3}(z)$ belongs to the stable subspace of $f^{-t_4-t_3}(x)$. Note that the support of this perturbation must be disjoint from $\omega$; this is the reason why $z$ must be ``far enough from $x$''. Finally, we find another time $t_2$ such that the negative orbit $\{f^{-t}(z')\}_{t\ge 0}$ of $z' = f^{-t_3-t_2}(z)$ has an hyperbolic behaviour. We then perturb each point of the negative orbit of $z'$ (within the stable manifold of $\omega$), so that it contains an arbitrary point of the stable manifold of $\omega$, far enough from $\omega$. This allows us to meet the point $y'$, provided 
that the order of discretizations $N$ is large enough.

To complete the proof, we we consider the segment of $f$-orbit joining $y$ to $z$; we perturb each one of these points to put them on the grid $E_N$ (with a perturbation whose supports size is proportional to $1/N$).

Notice that we shall have chosen carefully the parameters of the first perturbations in order to make this final perturbation possible. Also, remark that the length of the periodic orbit $\omega$ must be very large compared to the times $t_2$, $t_3$ and $t_4$. This is why we will perform the proof in the opposite direction : we will begin by choosing the times $t_i$ and make the perturbation of the dynamics afterwards.

Note that the Addendum~\ref{AddTheoMesPhysDiff} can be proved by using a small variation on Lemma~\ref{LemMesPhysDiff}, that we will explain at the end of Section~\ref{AvDerSec}.

\section{Discretizations of sequences of linear maps}\label{SecLin}

We begin by the study of the corresponding linear case, corresponding to the ``local behaviour'' of $C^1$ maps. We first define the linear counterpart of the discretization.

\begin{definition}\label{DefDiscrLin}
The map $P : \R\to\Z$\index{$P$} is defined as a projection from $\R$ onto $\Z$. More precisely, for $x\in\R$, $P(x)$ is the unique\footnote{Remark that the choice of where the inequality is strict and where it is not is arbitrary.} integer $k\in\Z$ such that $k-1/2 < x \le k + 1/2$. This projection induces the map\index{$\pi$}
\[\begin{array}{rrcl}
\pi : & \R^n & \longmapsto & \Z^n\\
 & (x_i)_{1\le i\le n} & \longmapsto & \big(P(x_i)\big)_{1\le i\le n}
\end{array}\]
which is an Euclidean projection on the lattice $\Z^n$. Let $A\in M_n(\R)$. We denote by $\widehat A$ the \emph{discretization}\index{$\widehat A$} of the linear map $A$, defined by 
\[\begin{array}{rrcl}
\widehat A : & \Z^n & \longrightarrow & \Z^n\\
 & x & \longmapsto & \pi(Ax).
\end{array}\]
\end{definition}

This definition allows us to define the rate of injectivity for sequences of linear maps.

\begin{definition}\label{DefTaux}
Let $A_1,\cdots,A_k \in GL_n(\R)$. The \emph{rate of injectivity} of $A_1,\cdots,A_k$ is the quantity\footnote{By definition, $B_R = B_\infty(0,R)\cap \Z^n$.}\footnote{In the sequel we will see that the $\limsup$ is in fact a limit.}\index{$\tau^k$}
\[\tau^k(A_1,\cdots,A_k) = \limsup_{R\to +\infty} \frac{\card \big((\widehat{A_k}\circ\cdots\circ\widehat{A_1})(\Z^n) \cap B_R\big)}{\card \big( \Z^n \cap B_R \big)}\in]0,1],\]
and for an infinite sequence $(A_k)_{k\ge 1}$ of invertible matrices, as the previous quantity is decreasing in $k$, we can define the \emph{asymptotic rate of injectivity}\index{$\tau^\infty$}
\[\tau^\infty\big((A_k)_{k\ge 1}\big) = \lim_{k\to +\infty}\tau^k(A_1,\cdots,A_k)\in[0,1].\]
\end{definition}

Finally, we define a topology on the set of sequences of linear maps.

\begin{definition}\label{DefTopoSL}
We fix once for all a norm $\|\cdot\|$ on $M_n(\R)$. For a bounded sequence $(A_k)_{k\ge 1}$ of matrices of $SL_n(\R)$, we set\index{$\|(A_k)_k\|$}
\[\|(A_k)_k\|_\infty = \sup_{k\ge 1} \|A_k\|.\]
In other words, we consider the space $\ell^\infty(SL_n(\R))$ of uniformly bounded sequences of matrices of determinant 1, endowed with the metric $\|\cdot\|_\infty$.
\end{definition}

We can now state the result we are interested in.

\begin{lemme}\label{DerTheoPart2}
For every $R_0>0$ and $\delta>0$, there exists $k_0\in\N$ such that the set
$\mathcal O_\varep^{k_0}$ of sequences $\{(A_k)_{k\ge 1}\in\ell^\infty(SL_n(\R))$ such that there exists a sequence $(w_k)_{k\ge 1}$ of translation vectors belonging to $[-1/2,1/2]^n$, and a vector $\widetilde y_0\in\Z^n$, with norm bigger than $R_0$, such that ($\pi(A+w)$\index{$\pi(A+w)$} denotes the discretization of the affine map $A+w$)
\[\big(\pi(A_{k_0} + w_{k_0})\circ \cdots \circ \pi(A_1 + w_1)\big)(\widetilde y_0) = \big(\widehat{A_{k_0} + w_{k_0}}\circ \cdots \circ \widehat{A_1 + w_1}\big)(0) =0.\]
Moreover, the point $\widetilde y_0$ being fixed, this property can be supposed to remain true on a whole neighbourhood of the sequence $(A_k)_{k\ge 1}\in\mathcal O_\varep^{k_0}$.
\end{lemme}

This lemma will be deduced from the following one.

\begin{lemme}\label{ConjPrincip}
For a generic sequence of matrices $(A_k)_{k\ge 1}$ of $\ell^\infty(SL_n(\R))$, we have
\[\tau^\infty\big( (A_k)_{k\ge 1}\big) = 0.\]
Moreover, for every $\varep>0$, the set of $(A_k)_{k\ge 1}\in\ell^\infty(SL_n(\R))$ such that  $\tau^\infty\big( (A_k)_{k\ge 1}\big) <\varep$ contains an open and dense subset of $\ell^\infty(SL_n(\R))$.
\end{lemme}

\begin{rem}
The second part of this statement is easily deduced from the first by applying the continuity of $\tau^k$ on a generic subset (Remark~\ref{conttaukk}).
\end{rem}

\begin{rem}
The same statement holds for generic sequences of isometries; this leads to nice applications to image processing (see \cite{Gui15b}).
\end{rem}

To prove Lemma~\ref{ConjPrincip}, we take advantage of the rational independence between the matrices of a generic sequence to obtain geometric formulas for the computation of the rate of injectivity. The tool used to do that is the formalism of \emph{model sets}\footnote{Also called \emph{cut-and-project} sets.} (see for example \cite{Moody25} or \cite{MR2876415} for surveys about model sets, see also \cite{Gui15d} for the application to the specific case of discretizations of linear maps). Lemma~\ref{ConjPrincip} is proved in \cite{Gui15a}. However, we have chosen to include a condensed proof in appendix for the sake of completeness. For more details and comments about the linear case, see also \cite{Guih-These}.\label{model}

We now explain how to deduce Lemma~\ref{DerTheoPart2} from Lemma~\ref{ConjPrincip}.

\begin{proof}[Proof of Lemma~\ref{DerTheoPart2}]
We set
\[ \mathcal O_\varep^k = \{(A_k)_{k\ge 1}\in\ell^\infty(SL_n(\R)) \mid \tau^k(A_1,\cdots,A_k) <\varep\}.\]
Lemma~\ref{ConjPrincip} states that for every $\varep>0$, the set $\bigcup_{k\ge 0} \mathcal O_\varep^k$ contains an open and dense subset of $\ell^\infty(SL_n(\R))$. Together with the continuity of $\tau^k$ at every generic sequence (Remark~\ref{conttaukk}), this implies that for every $\delta>0$, there exists $k_0>0$ such that $\mathcal O_\varep^{k_0}$ contains an open and $\delta$-dense subset of $\ell^\infty(SL_n(\R))$.

Then, if $\tau^{k_0}(A_1,\cdots,A_k) <\varep$, then there exists a point $x_0\in\Z^n$ such that
\[\card\big((A_{k_0}\circ\cdots\circ A_1)^{-1}(x_0)\big) \ge \frac{1}{\varep}\]
(and moreover if the sequence $(A_k)_{k\ge 1}$ is generic, then this property remains true on a whole neighbourhood of the sequence). The lemma follows from this statement by remarking that on the one hand, if we choose $w_k\in [-1/2,1/2]^n$ such that
\[w_k = A_k^{-1}\Big(\big(\widehat{A_{k_0}}\circ \cdots \circ \widehat{A_{k-1}}\big)^{-1}(x_0)\Big) \mod \Z^n,\]
then the properties of the cardinality of the inverse image of $x_0$ are transferred to the point 0, and that on the other hand, for every $R_0>0$, there exists $m\in\N$ such that every subset of $\Z^n$ with cardinality bigger than $m$ contains at least one point with norm bigger than $R_0$.
\end{proof}

\section{An improved ergodic closing lemma}

The proof of Theorem~\ref{TheoMesPhysDiff} begins by the approximation of any invariant measure $\mu$ of any conservative $C^1$-diffeomorphism by a periodic measure of a diffeomorphism $g$ close to $f$. This is done by R.~Mañé's ergodic closing lemma, but we will need the fact that the obtained periodic measure inherits some of the properties of the measure $\mu$. More precisely, given a $C^1$-diffeomorphism $f$, we will have to approach any non periodic ergodic measure of $f$ with nonzero Lyapunov exponent by a periodic measure of a diffeomorphism $g$ close to $f$, such that the Lyapunov exponents and the Lyapunov subspaces of the measure are close to that of $f$ by $\mu$. We will obtain this result by modifying slightly the proof of a lemma obtained by F.~Abdenur, C.~Bonatti and S.~Crovisier in \cite{MR2811152} (Proposition~6.1).

\begin{lemme}[Ergodic closing lemma]\label{ErgoLemPlus}%\Argh{Enoncer en toute généralité !}
Let $f\in\Diff^1(\T^n,\Leb)$. We consider
\begin{itemize}
\item a number $\varep>0$;
\item a $C^1$-neighbourhood $\mathcal V$ of $f$;
\item a time $\tau_0\in\N$;
\item an ergodic measure $\mu$ without zero Lyapunov exponent;
\item a point $x\in \T^n$ which is typical for $\mu$ (see the beginning of the paragraph 6.1 of \cite{MR2811152});
\end{itemize}
moreover, we denote by $\lambda$ the smallest absolute value of the Lyapunov exponents of $\mu$, by $F^f_x$ the stable subspace at $x$ and by $G_x^f$ the unstable subspace\footnote{Stable and unstable in the sense of Oseledets splitting.} at $x$.
Then, there exists a diffeomorphism $g\in\Diff^1(\T^n,\Leb)$ and a time $\widetilde{t_0}>0$ (depending only in $f$, $\mu$ and $x$) such that:
\begin{enumerate}
\item $g\in\mathcal V$;
\item the point $x$ is periodic for $g$ of period $\tau\ge \tau_0$;
\item for any $t\le\tau$, we have $d\big(f^t(x),g^t(x)\big)<\varep$;
\item $x$ has no zero Lyapunov exponent for $g$ and the smallest absolute value of the Lyapunov exponents of $x$ is bigger than $\lambda/2$, we denote by $F^g_x$ the stable subspace and $G^g_x$ the unstable subspace;
\item the angles between $F^f_x$ and $F^g_x$, and between $G^f_x$ and $G^g_x$, are smaller than $\varep$;
\item for any $t\ge \widetilde{t_0}$, for any vectors of unit norm $v_F\in F^g_x$ and $v_G\in G^g_x$, we have
\[\frac{1}{t}\log \big(\|Dg^{-t}_x (v_F)\|\big)\ge \frac{\lambda}{4} \qquad \text{and} \qquad \frac{1}{t}\log \big(\|Dg^{t}_x (v_G)\|\big)\ge \frac{\lambda}{4}\]
\end{enumerate}
\end{lemme}

Remark that the proof of Proposition~6.1 of \cite{MR2811152} yields a similar lemma but with the weaker conclusion
\emph{\begin{itemize}
\item[5.] ``the angle between $G^f_x$ and $G^g_x$, is smaller than $\varep$''.
\end{itemize}}
\noindent Indeed, the authors obtain the linear space $G^g_x$ by a fixed point argument: Lemma~6.5 of \cite{MR2811152} states that the cone $C^s_{j,4C}$ is invariant by $Df_n^{-t_n}$, and thus contains both $G^f_x$ and $G^g_x$. Taking $C$ as big as desired, the cone $C^s_{j,4C}$ is as thin as desired and thus the angle between $G^f_x$ and $G^g_x$, is as small as desired. Unfortunately, in the original proof of Proposition~6.1 of \cite{MR2811152}, the linear space $F^g_x$ is not defined in the same way ; it is an invariant subspace which belongs to $C^u_{j,4C}$, which is an arbitrarily thick cone. Thus, the angle between $F^f_x$ and $F^g_x$, is not bounded by this method of proof. Our goal here is to modify the proof of Proposition~6.1 of \cite{MR2811152} to have simultaneously two thin cones $C'^u_{j,4C}$ and $C^s_{j,4C}$ which are invariant under respectively $Df_n^{t_n}$ and $Df_n^{-t_n}$

We begin by modifying the Lemma~6.2 of \cite{MR2811152}: we replace its forth point
\begin{itemize}
\item \emph{a sequence of linear isometries $P_n\in O_d(\R)$ such that $\|P_n - \Id\|< \varep$,}
\end{itemize}
by the point
\begin{itemize}
\item \emph{two sequences of linear isometries $P_n,Q_n\in O_d(\R)$ such that $\|P_n - \Id\|< \varep$ and $\|Q_n - \Id\|< \varep$,}
\end{itemize}
and its forth conclusion
\begin{itemize}
\item[\emph{d)}] \emph{For every $i\le j\in\{1,\cdots,k\}$ the inclination\footnote{The \emph{inclination} of a linear subspace $E\subset\R^n$ with respect to another subspace $E'\subset\R^n$ with the same dimension is the minimal norm of the linear maps $f : E\to E^\perp$ whose graph are equal to $E$.} of $Df_n^{t_n}. E_{i,j}$ with respect to $E_{i,j}$ is less than $C$.}
\end{itemize}
by the conclusion
\begin{itemize}
\item[\emph{d)}] \emph{For every $i\le j\in\{1,\cdots,k\}$ the inclination of $Df_n^{t_n}. E_{i,j}$ with respect to $E_{i,j}$ is less than $C$, and the inclination of $Df_n^{-t_n}. E_{i,j}$ with respect to $E_{i,j}$ is less than $C$.}
\end{itemize}

These replacements in the lemma are directly obtained by replacing Claim~6.4 of \cite{MR2811152} by the following lemma.

\begin{lemme}\label{Claim6.4}
For any $\eta>0$, there exists a constant $C>0$ such that for any matrix $A\in GL_n(\R)$ and any linear subspace $E\subset\R^n$, there exists two orthogonal matrices $P,Q\in O_n(\R)$ satisfying $\|P-\Id\|<\eta$ and $\|Q-\Id\|<\eta$, such that the inclinations of $(PAQ)(E)$ and $(PAQ)^{-1}(E)$ with respect to $E$ are smaller than $C$.
\end{lemme}

%\begin{lemme}\label{Claim6.4}
%For any $\eta>0$, there exists a constant $C>0$ such that for any matrix $A\in GL_n(\R)$ and any decomposition $\R^n = E\oplus F$, there exists two orthogonal matrices $P,Q\in O_n(\R)$ satisfying $\|P-\Id\|<\eta$ and $\|Q-\Id\|<\eta$, such that the inclinations of $(PAQ)(E)$ and $(PAQ)^{-1}(E)$ with respect to the decomposition $\R^n = E\oplus F$ are smaller than $C$.
%\end{lemme}

\begin{proof}[Proof of Lemma~\ref{Claim6.4}]
Given $\eta>0$, there exists a constant $C>0$ and a matrix $P_0\in O_n(\R)$ such that $\|P_0-\Id\|<\eta$, satisfying: for any linear subspace $E'\subset \R^n$, one of the two inclinations of $E'$ and of $P_0(E')$ with respect to $E$ is smaller then $C$.

We then choose an orthogonal matrix $Q\in O_n(\R)$ such that $\|Q-\Id\|<\eta$ and that (taking a bigger $C$ if necessary) both inclinations of $Q^{-1}\big(A^{-1}(E)\big)$ and $Q^{-1}\big((A^{-1} P_0^{-1})(E)\big)$ with respect to $E$ are smaller than $C$. There are two cases: either the inclination of $(AQ)(E)$ with respect to $E$ is smaller than $C$, and in this case we choose $P=\Id$, or the inclination of $(AQ)(E)$ with respect to $E$ is bigger than $C$, and in this case we can choose $P=P_0$. In both cases, the lemma is proved.
\end{proof}

The rest of the proof of Lemma~\ref{ErgoLemPlus} can be easily adapted from the proof of Proposition~6.1 of \cite{MR2811152}.

\section[Proof of the perturbation lemma]{Proof of the perturbation lemma (Lemma~\ref{LemMesPhysDiff})}\label{AvDerSec}

We now come to the proof of Lemma~\ref{LemMesPhysDiff}. We first do this proof in dimension 2, to simplify some arguments and to be able to make pictures.

\begin{proof}[Proof of Lemma~\ref{LemMesPhysDiff}]
Let $f$ be a conservative $C^1$-diffeomorphism, $\mathcal V$ a $C^1$-neighbourhood of $f$, $\varep>0$ and $N_0\in\N$. We denote $M = \max\big(\|Df\|_\infty, \|Df^{-1}\|_\infty\big)$. We also choose an $f$-invariant measure $\mu$ which is ergodic, not periodic and has no zero Lyapunov exponent, and an open set $U\subset \T^2$. We will make several successive approximations of $f$ in $\mathcal V$; during the proof we will need to decompose this neighbourhood: we choose $\delta>0$ such that the open $\delta$-interior $\mathcal V'$ of $\mathcal V$ is non-empty.

\paragraph{Step 0: elementary perturbation lemmas.}
During the proof of Lemma~\ref{LemMesPhysDiff}, we will use three different elementary perturbation lemmas.

The first one is the elementary perturbation lemma in $C^1$ topology.

\begin{lemme}[Elementary perturbation lemma in $C^1$ topology]\label{PerturbElem}
For every diffeomorphism $f\in\Diff^1(\T^n,\Leb)$ and every $\delta>0$, there exists $\eta>0$ and $r_0>0$ such that the following property holds: for every $x,y\in \T^n$ such that $d(x,y)<r_0$, there exists a diffeomorphism $g\in\Diff^1(\T^n,\Leb)$ satisfying $d_{C^1}(f,g)<\delta$, such that $g(x) = f(y)$ and that $f$ and $g$ are equal out of the ball $B\big(\frac{x+y}{2},\frac{1+\eta}{2}d(x,y)\big)$.
\end{lemme}

This lemma allows to perturb locally the orbit of a diffeomorphism; a proof of it can be found for example in \cite[Proposition~5.1.1]{MR1662930}. 

The second one is an easy corollary of the first one. We will use it to perturb a segment of orbit such that for any $N$ large enough, each point of this segment of orbit belongs to the grid $E_N$.

\begin{lemme}[Perturbation of a point such that it belongs to the grid]\label{Lem*}
For every open set $\mathcal V'$ of $\Diff^1(\T^n,\Leb)$, there exists $\eta'>0$ such that for $N$ large enough an every $x\in\T^n$, there exists $g\in\Diff^1(\T^n,\Leb)$ such that
\begin{itemize}
\item $g\in\mathcal V'$;
\item $g(x_N) = \big( f(x)\big)_N$;
\item $f=g$ outside of $B\big(x,(1+\eta')/N\big)$.
\end{itemize} 
\end{lemme}

Applying this lemma to several points $x_i\in \T^n$ which are far enough one from the others (for $i\neq j$, $d(x_i,x_j)\ge 2(1+\eta')/N$), it is possible to perturb $f$ into a diffeomorphism $g$ such that for every $i$, $g\big((x_i)_N\big) = \big( f(x_i)\big)_N$.

These two perturbations will be applied locally.
\bigskip

The third perturbation lemma is an improvement of Lemma~\ref{PerturbElem}; it states that the perturbation can be supposed to be a translation in a small neighbourhood of the perturbed point.

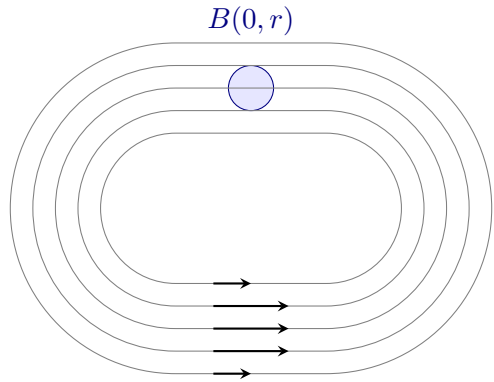
\begin{figure}[t]
\begin{center}
\begin{tikzpicture}[scale=1]

\fill[color=blue!10!white] (0,1.6) circle (.3);
\draw[color=blue!50!black] (0,1.6) circle (.3);
\draw[color=blue!50!black] (0,2.5) node{$B(0,r)$};

\draw[color=gray] (1,-1) arc (-90:90:1) -- (-1,1) arc (90:270:1) -- cycle;
\draw[color=gray] (1,-1.3) arc (-90:90:1.3) -- (-1,1.3) arc (90:270:1.3) -- cycle;
\draw[color=gray][color=gray] (1,-1.6) arc (-90:90:1.6) -- (-1,1.6) arc (90:270:1.6) -- cycle;
\draw[color=gray] (1,-1.9) arc (-90:90:1.9) -- (-1,1.9) arc (90:270:1.9) -- cycle;
\draw[color=gray] (1,-2.2) arc (-90:90:2.2) -- (-1,2.2) arc (90:270:2.2) -- cycle;
\draw[->,>=stealth,thick] (-.5,-1) -- (0,-1);
\draw[->,>=stealth,thick] (-.5,-1.3) -- (.5,-1.3);
\draw[->,>=stealth,thick] (-.5,-1.6) -- (.5,-1.6);
\draw[->,>=stealth,thick] (-.5,-1.9) -- (.5,-1.9);
\draw[->,>=stealth,thick] (-.5,-2.2) -- (0,-2.2);

\end{tikzpicture}
\caption[Proof of Lemma~\ref{LocTrans}]{Flow of the Hamiltonian used to prove Lemma~\ref{LocTrans} (``staduim'').}\label{FigStade}
\end{center}
\end{figure}

\begin{lemme}[Elementary perturbation with local translation]\label{LocTrans}
For every open set $\mathcal V'$ of $\Diff^1(\T^n,\Leb)$, and every $r>0$, there exists $N_1>0$ such that for every $N\ge N_1$ and every $\|v\|_\infty \le 1/(2N)$, there exists $g\in\Diff^1(\R^n,\Leb)$ such that:
\begin{itemize}
\item $g\in\mathcal V'$;
\item $\operatorname{Supp}(g)\subset B(0,10\,r)$;
\item for every $x\in B(0,r)$, $g(x) = x+v$.
\end{itemize}
\end{lemme}

\begin{proof}[Proof of Lemma~\ref{LocTrans}]
Take an appropriate Hamiltonian, see Figure~\ref{FigStade}.
\end{proof}

\paragraph{Step 1: choice of the starting point $x$ of the orbit.} Let $\lambda$ be the smallest absolute value of the Lyapunov exponents of $\mu$ (in particular, $\lambda > 0$).

We choose a point $x$ which is regular for the measure $\mu$: we suppose that it satisfies the conclusions of Oseledets and Birkhoff theorems, and Mañé's ergodic closing lemma (see Paragraph~6.1 of \cite{MR2811152}). We denote by $F_x^f$ the stable subspace and $G_x^f$ the unstable subspace for the Oseledets splitting at the point $x$. By Oseledets theorem, the growth of the angles $\angle\big(F_{f^i(x)}^f,G_{f^i(x)}^f\big)$ between the stable and unstable subspaces is subexponential (in both positive and negative times).

\paragraph{Step 2: choice of the parameters we use to apply the ergodic closing lemma.} In this second step, we determine the time during which we need an estimation of the angle between the stable and unstable subspaces of $f$ and its perturbations, and the minimal length of the approximating periodic orbit.

We first use the ``hyperbolic-like'' behaviour of $f$ near the orbit of $x$: for well chosen times $t_1$ and $t_2$, each vector which is not too close to $G_{f^{t_1}(x)}^f$ is mapped by $Df^{-t_2}$ into a vector which is close to $F_{f^{t_1-t_2}(x)}^f$. Given a vector $v\in T\T^n_{f^{t_1}(x)}$, it will allow us to perturb $f$ into $g$ such that an iterate of $v$ under $Dg^{-1}$ belongs to $F_{f^{t_1-t_2}(x)}^f$.

\begin{lemme}\label{LemAnglesOsel}
For every $\alpha>0$, there exists two times $t_1$ and $t_2 \ge 0$ such that if $v\in T \T^n_{f^{t_1}(x)}$ is such that the angle between $v$ and $G_{f^{t_1}(x)}$ is bigger than $\alpha$, then the angle between $Df^{-t_2}_{f^{t_1}(x)} v$ and $F_{f^{t_1 - t_2}(x)}$ is smaller than $\alpha$ (see Figure~\ref{FigRotaGlobC1}).
\end{lemme}

\begin{proof}[Proof of Lemma~\ref{LemAnglesOsel}]
It easily follows from Oseledets theorem, and more precisely from the fact that the function $\exp(t\lambda) / \angle(F_{f^t(x)},G_{f^t(x)})$ goes to $+\infty$ when $t$ goes to $+\infty$.
\end{proof}

So, we fix two times $t_1$ and $t_2 \ge 0$, obtained by applying Lemma~\ref{LemAnglesOsel} to $\alpha = \arcsin\big(1/(1+\eta)\big)$, where $\eta$ is the parameter obtained by applying the elementary perturbation lemma (Lemma~\ref{PerturbElem}) to $\delta/2$ (see Figure~\ref{FigRotaC1}).

We also choose a time $t_3\ge \widetilde t_0$ ($\widetilde t_0$ being given by Lemma~\ref{ErgoLemPlus}) such that
\begin{equation*}
e^{\lambda (t_3+t_2)/4} \ge M^{t_2}.
\end{equation*}
This estimation will be applied to point 6. of Lemma~\ref{ErgoLemPlus}. It will imply that for every $v\in F_{f^{t_1}(x)}^f$ and for every $t\ge t_2+t_3$, we have
\begin{equation}\label{Pastropgros}
\|Df_{f^{t_1}(x)}^{-t}(v)\| \ge \|Df_{f^{t_1}(x)}^{-t_2}(v)\| \ge \frac{1}{M^{t_2}} \|v\|.
\end{equation}
We then apply Lemma~\ref{DerTheoPart2} to
\begin{equation}\label{DefRr0}
R_0 = M^{t_2+t_3}(1+\eta'),
\end{equation}
where $\eta'$ is given by Lemma~\ref{Lem*} applied to the parameter $\delta/2$. This gives us a parameter $k_0 = t_4$. Note that $R_0$ is chosen so that if $v\in T_{f^{t_1}(x)}\T^n$ is such that $\|v\|\ge R_0/N$, then for any $t\in \llbracket 0, t_2+t_3\rrbracket$, we have
\begin{equation}\label{EqPasIdee}
\big\|Df^{-t}_{f^{t_1}(x)}(v)\big\| \ge (1+\eta')/N.
\end{equation}
Thus, we will be able to apply Lemma~\ref{Lem*} to the points $f^{-t}\big(f^{t_1}(x) + v\big)$, with $t\in \llbracket 0, t_2+t_3\rrbracket$, without perturbing the points of the orbit of $x$.

\paragraph{Step 3: global perturbation of the dynamics.} We can now apply the ergodic closing lemma we have stated in the previous section (Lemma~\ref{ErgoLemPlus}) to the neighbourhood $\mathcal V'$, the measure $\mu$, the point $x_1 = f^{t_1-t_2-t_3}(x)$ and $\tau_0 \ge t_2+t_3+t_4$ large enough so that $\tau_0 \lambda/4\ge 3$. We also need that the expansion of vectors $F^{g_1}$ along the segment of orbit $\big(x_2,g_2(x_2),\cdots,g_2^{\tau_0-t_2-t_3-t_4}(x_2)\big)$ is bigger than 3, but it can be supposed true by taking a bigger $\tau_0$ if necessary. This gives us a first perturbation $g_1$ of the diffeomorphism $f$, such that the point $x_1$ is periodic under $g_1$ with period $\tau_1\ge\tau_0$, and such that the Lyapunov exponents of $x_1$ for $g_1$ are close to that of $x_1$ under $f$, and the stable and unstable subspaces of $g_1$ at the point $g_1^t(x_1)$ are close to that of $f$ at the point $g_1^t(x_1)$ for every $t\in \llbracket 0,t_3+t_2\rrbracket$.

Remark that by the hypothesis on $\tau_0$, the Lyapunov exponent of $g_1^{\tau_1}$ at $x_1$ is bigger than 3, thus we will be able to apply Lemma~\ref{Lem*} to every point of the orbit belonging to $F^{g_1}_{x_1}$, even when the orbit returns several times near $x_1$. Also note that these properties are stable under $C^1$ perturbation.

We then use the connecting lemma for pseudo-orbits of C.~Bonatti and S.~Crovisier (see \cite{MR2090361}), which implies that the stable manifolds of the periodic orbits of a generic conservative $C^1$-diffeomorphism are dense. This allows us to perturb the diffeomorphism $g_1$ into a diffeomorphism $g_2\in\mathcal V'$ such that there exists a point $x_2$ close to $x_1$ such that:
\begin{enumerate}[(1)]
\item $x_2$ is periodic for $g_2$ with the same period than that of $x_1$ under $g_1$, and moreover the periodic orbit of $x_2$ under $g_2$ shadows that of $x_1$ under $g_1$;
\item the Lyapunov exponents and the Lyapunov subspaces of $x_2$ for $g_2$ are very close to that of $x_1$ for $g_1$ (see the conclusions of Lemma~\ref{ErgoLemPlus}, in particular the Lyapunov subspaces are close during a time $t_3+t_2$);
\item the stable manifold of $x_2$ under $g_2$ meets the set $U$, at a point denoted by $y_2$.
\end{enumerate}

\paragraph{Step 4: linearization near the periodic orbit.} We then use Franks lemma (see \cite{MR0283812}) to perturb slightly the differentials of $g_2$ at the points $g_2^{t_2+ t_3}(x_2),$ $\cdots,g_2^{t_2+ t_3+t_4}(x_2)$, such that these differentials belong to the open set of matrices defined by Lemma~\ref{DerTheoPart2}. This gives us another diffeomorphism $g_3\in\mathcal V'$ close to $g_2$, such that the point $x_2$ still satisfies the nice properties (1), (2) and (3).

We then use a lemma of \cite{ArturSylvain} which allows to linearize locally a conservative diffeomorphism.

\begin{lemme}[Avila, Crovisier, Wilkinson]\label{LemExtension}
Let $C$ be the unit ball of $\R^n$ for $\|\cdot\|_\infty$ and $\varep>0$. Then, there exists $\delta>0$ such that for every $g_1\in\Diff^\infty(\R^n,\Leb)$ such that $d_{C^1}({g_1}_{|C}, \Id_{|C})<\delta$, there exists $g_2\in\Diff^\infty(\R^n,\Leb)$ such that:
\begin{enumerate}[(i)]
\item $d_{C^1}({g_2}_{|C}, {g_1}_{|C})<\varep$;
\item ${g_2}_{|(1-\varep)C} = \Id_{|(1-\varep)C}$;
\item ${g_2}_{|C^\complement} = {g_1}_{|C^\complement}$.
\end{enumerate}
\end{lemme}

The proof of this lemma involves a result of J.~Moser \cite{MR0182927}. The reader may refer to \cite[Corollary 6.9]{ArturSylvain} for a complete proof\footnote{The 10/01/2015, this version is not published online yet\dots}. By a regularization result due to A.~Avila \cite{MR2736152}, it is possible to weaken the hypothesis of regularity in the lemma ``$g_1\in\Diff^\infty(\R^n,\Leb)$'' into the hypothesis ``$g_1\in\Diff^1(\R^n,\Leb)$''.

By Lemma~\ref{LemExtension}, there exists a parameter $r>0$ such that it is possible to linearize $g_3$ in the $r$-neighbourhood of the periodic orbit of $x_2$, without changing the nice properties (1), (2) and (3) of the periodic orbit of $x_2$. We can choose $r$ small enough so that the $10\,r$-neighbourhoods of the points of the periodic orbit of $x_2$ are pairwise disjoint. This gives us a diffeomorphism $g_4$, to which are associated two points $x_4$ and $y_4$, such that $x_4$ satisfies the properties (1), (2) and (3), and such that:
\begin{enumerate}[(1)]\setcounter{enumi}{3}
\item the differentials of $f$ at the points $g_4^{t_2+ t_3}(x_4),\cdots,g_4^{t_2+ t_3+t_4}(x_4)$ lie in the open dense set of matrices of Lemma~\ref{DerTheoPart2};
\item $g_4$ is linear in the $r$-neighbourhood of each point of the periodic orbit of $x_4$.
\end{enumerate}

\paragraph{Step 5: choice of the order of discretization.}

We choose a neighbourhood $\mathcal V''\subset\mathcal V'$ of $g_4$ such that properties (1) to (3) are still true for every diffeomorphism $g\in\mathcal V''$. We denote by $\omega_{x_4}$ the periodic orbit of $x_4$ under $g_4$, and by $B(\omega_{x_4},r)$ the $r$-neighbourhood of this periodic orbit. We also denote $T_1$ the smallest integer such that $g_4^t(y_4) \in B(\omega_{x_4},r/2)$ for every $t\ge T_1$, and set $y'_4 = g_4^{T_1}(y_4)$. Thus, the positive orbit of $y'_4$ will stay forever in the linearizing neighbourhood of $\omega_{x_4}$. Taking $T_1$ bigger if necessary, we can suppose that $y'_4$ belongs to the linearizing neighbourhood of the point $x_4$. We can also suppose that for every $t\in \llbracket 0, \tau_1\rrbracket$,
\begin{equation}\label{LoinOrbPer}
3d\big(g_4^{T_1-t}(y_4),g_4^{-t}(x_4)\big) \le \min_{\tau_1\le t'\le T_1}d\big(g_4^{T_1-t'}(y_4),g_4^{-t}(x_4)\big).
\end{equation}

We can now choose the order $N$ of the discretization, such that
\begin{enumerate}[(i)]
\item $N\ge N_0$ ($N_0$ has been chosen at the very beginning of the proof);
\item \label{itemlinea} applying Lemma~\ref{LocTrans} to the parameter $r$ and the neighbourhood $\mathcal V''$ to get an integer $N_1$, we have $N\ge N_1$, so that it is possible to choose the value of the points of $\omega_{x_4}$ modulo $E_N$ without changing the properties (1) to (5);
\item the distance between two distinct points of the segment of orbit $y_4,$ $g_4(y_4), \cdots, g_4^{T_1}(y_4) = y'_4$ is bigger than $2(1+\eta')/N + 2/N$, so that it will be possible to apply Lemma~\ref{Lem*} simultaneously to each of these points, even after the perturbation made during the point \eqref{itemlinea}, such that these points belong to $E_N$;
\item every $\sqrt 2/N$-pseudo-orbit\footnote{The constant $\sqrt n/N$ comes from the fact that an orbit of the discretization is a $\sqrt 2/N$-pseudo-orbit.} starting at a point of the periodic orbit $\omega_{x_4}$ stays during a time $T'\tau_1$ in the $d(y'_4,\omega_{x_4})$-neighbourhood of the periodic orbit, where $T'$ the smallest integer such that
\begin{equation}\label{hyphyphyp}
 \left(1 + \frac{1}{3(1+\eta)}\right)^{T'} \ge \nu,
\end{equation}
and $\nu$ is the maximal modulus of the eigenvalues of $(Dg_4)_{x_4}^{\tau_1}$. A simple calculus shows that this condition is true if for example
\begin{equation*}
N \ge \frac{2\sqrt n (M^{T'\tau_1}-1)}{r(M-1)}.
\end{equation*}
This condition will be used to apply the process described by Lemma~\ref{LemLyapPerturb}.
\end{enumerate}

\paragraph{Step 6: application of the linear theorem.}

By the hypothesis (ii) on $N$, we are able to use Lemma~\ref{LocTrans} (elementary perturbation with local translation) to perturb each point of the periodic orbit $\omega_{x_4}$ such that we obtain a diffeomorphism $g_5\in\mathcal V''$ and points $x_5$, $y_5$ and $y'_5$ satisfying properties (1) to (5) and moreover:
\begin{enumerate}[(1)]\setcounter{enumi}{5}
\item for every $t\in \llbracket t_2 + t_3, t_2 + t_3 + t_4\rrbracket$, the value of $g_5^t(x_5)$ modulo $E_N$ is equal to $w_k/N$, where $w_k$ is given by Lemma~\ref{DerTheoPart2};
\item for any other $t$, $g_5^t(x_5)$ belongs to $E_N$.
\end{enumerate}
In particular, the periodic orbit of $x_5$ under $g_5$ is stabilized by the discretization $(g_5)_N$ (indeed, recall that $w_k \in [-1/2,1/2]^k$).

By construction of the diffeomorphism $g_5$ (more precisely, the hypotheses (4), (5), (6) and (7)), it satisfies the conclusions of Lemma~\ref{DerTheoPart2}; thus there exists a point $z\in B\big(g_5^{t_2+t_3}(x_5),r\big)$ such that $(g_5)_N^{t_4} (z) = (g_5)_N^{t_2+t_3+t_4} (x_5)$ and that $\|z - g_5^{t_2+t_3}(x_5)\|\ge R_0/N$ (where $R_0$ is defined by Equation~\eqref{DefRr0}). Remark that hypothesis (iv) implies that $\|z-g_5^{t_2+t_3}(x_5)\|\ll r$.

\paragraph{Step 7: perturbations in the linear world.}

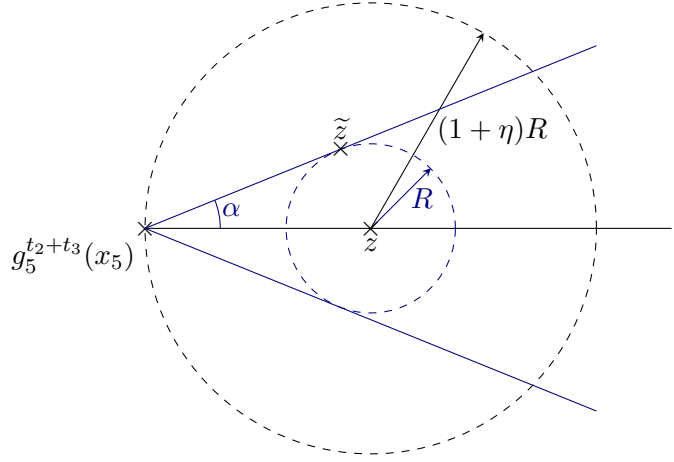
\begin{figure}[t]
\begin{center}
\begin{tikzpicture}[scale=1]
\draw (0,0) node {$\times$};
\draw (0,0) node[below left] {$g_5^{t_2+t_3}(x_5)$};

\draw (0,0) -- (7,0);
\draw[dashed] (3,0) circle (3);
\draw[->,>=stealth] (3,0) -- (4.5,2.6) node[pos=.5,right]{$(1+\eta)R$};
\draw[dashed, color=blue!50!black] (3,0) circle (9/8);
\draw[->,>=stealth, color=blue!50!black] (3,0) -- (3.8,0.8) node[pos=.5,right]{$R$};

\draw[color=blue!50!black] (0,0) -- (6,2.43);
\draw[color=blue!50!black] (0,0) -- (6,-2.43);
\draw[color=blue!50!black] (1,0) arc (0:22.05:1);
\draw[color=blue!50!black] (1.15,.25) node{$\alpha$};

\draw (2.6,1.053) node {$\times$};
\draw (2.6,1.053) node[above] {$\widetilde z$};

\draw (3,0) node {$\times$};
\draw (3,0) node[below] {$z$};
\end{tikzpicture}
\caption[Perturbation made to apply Lemma~\ref{LemAnglesOsel}]{Perturbation we make to apply Lemma~\ref{LemAnglesOsel} (see also Figure~\ref{FigRotaGlobC1}): we make an elementary perturbation in a neighbourhood of $z$ mapping $z$ into $\widetilde z$, such that the angle between the lines $\big(g_5^{t_2+t_3}(x_5)\ z\big)$ and $\big(g_5^{t_2+t_3}(x_5)\ \widetilde z\big)$ is bigger than $\alpha = \arcsin\big(1/(1+\eta)\big)$, and such that the support of the perturbation does not contain $g_5^{t_2+t_3}(x_5)$.}\label{FigRotaC1}
\end{center}
\end{figure}

In this step, our aim is to perturb the negative orbit of $z$ under $g_5$ such that it meets the point $y'_5$. Remark that by hypothesis (iv), every point of $z,g_5^{-1}(z),\cdots,g_5^{-t_2}(z)$ is in the linearizing neighbourhood of $\omega_{x_5}$.

From now, all the perturbations we will make will be local, and we will only care of the positions of a finite number of points. Thus, it will not be a problem if these perturbations make hypotheses (3) and (5) become false, provided that they have a suitable behaviour on this finite set of points.

First, if necessary, we make a perturbation in the way of Figure~\ref{FigRotaC1}, so that the angle between the lines $\big( g_5^{t_2+t_3}(x_5)\ z\big)$ and $G_{g_5^{t_2+t_3}(x_5)}^{g_5}$ is bigger than $\alpha$; this gives us a diffeomorphism $g_6$. More precisely, the support of the perturbation we apply is contained in a ball centred at $z$ and with radius $d(z,x_6)$, so that this perturbation does not change the orbit of $x_6$. Under these conditions, we satisfy the hypotheses of Lemma~\ref{LemAnglesOsel}, thus the angle between $\big( g_6^{t_3}(x_6)\ g_6^{-t_2}(z)\big)$ and $F_{{g_6}^{t_3}(x_6)}^{g_6}$ is smaller than $\alpha$. Another perturbation, described by Figure~\ref{FigRotaC1}, allows us to suppose that $g_6^{-t_2}(z)$ belongs to $F_{{g_6}^{t_3}(x_6)}^{g_6}$. This gives us a diffeomorphism that we still denote by $g_6$. Remark that it was possible to make these perturbations independently because the segment of negative orbit of the point $z$ we considered does not enter twice in the 
neighbourhood of a point of $\omega_{x_6}$ where the diffeomorphism is linear.

Thus, the points $z' = g_6^{-t_2}(z)$ and $y'_6 = y'_5$ both belong to the local stable manifold of the point $x_6 = x_5$ for $g_6$ (which coincides with the Oseledets linear subspace $F_{x_6}^{g_6}$ since $g_6$ is linear near $x_6$).
\bigskip

The next perturbation takes place in the neighbourhood of the point $x_6$ (and not in all the linearizing neighbourhoods of the points of $\omega_{x_6}$).

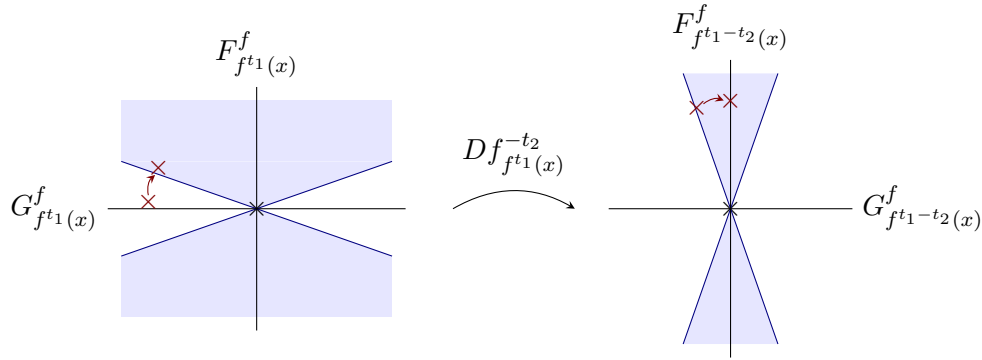
\begin{figure}[t]
\begin{center}
\begin{tikzpicture}[scale=.9]

\fill[color=blue!10!white] (-6,.7) -- (-2,-.7)  -- (-6,-.7) -- (-2,.7) -- cycle;
\fill[color=blue!10!white] (-6,.7) -- (-2,.7)  -- (-2,1.6) -- (-6,1.6) -- cycle;
\fill[color=blue!10!white] (-6,-.7) -- (-2,-.7)  -- (-2,-1.6) -- (-6,-1.6) -- cycle;
\draw (-4,0) node {$\times$};
\draw[color=blue!50!black] (-6,-.7) -- (-2,.7);
\draw[color=blue!50!black] (-6,.7) -- (-2,-.7);
\draw (-1.8,0) -- (-6.2,0);
\draw (-4,-1.8) -- (-4,1.8);
\draw (-4,1.8) node[above]{$F^f_{f^{t_1}(x)}$};
\draw (-6.2,0) node[left]{$G^f_{f^{t_1}(x)}$};
\draw[color=red!50!black] (-5.6,.1) node{$\times$};
\draw[color=red!50!black] (-5.45,.6) node{$\times$};
\draw[->,>=stealth, color=red!50!black] (-5.6,.2) to[bend left] (-5.5,.5);

\draw[->,>=stealth] (-1.1,0) to[bend left] (.7,0);
\draw (-.2,.8) node{$Df^{-t_2}_{f^{t_1}(x)}$};

\fill[color=blue!10!white] (2.3,-2) -- (3.7,2)  -- (2.3,2) -- (3.7,-2) -- cycle;
\draw (3,0) node {$\times$};
\draw[color=blue!50!black] (2.3,-2) -- (3.7,2);
\draw[color=blue!50!black] (2.3,2) -- (3.7,-2);
\draw (1.2,0) -- (4.8,0);
\draw (3,-2.2) -- (3,2.2);
\draw (3,2.2) node[above]{$F^f_{f^{t_1-t_2}(x)}$};
\draw (4.8,0) node[right]{$G^f_{f^{t_1-t_2}(x)}$};
\draw[color=red!50!black] (2.5,1.5) node{$\times$};
\draw[color=red!50!black] (3,1.6) node{$\times$};
\draw[->,>=stealth, color=red!50!black] (2.6,1.55) to[bend left] (2.9,1.6);

\end{tikzpicture}
\caption[Proof of Lemma~\ref{LemAnglesOsel}]{Proof of Lemma~\ref{LemAnglesOsel}: make a small perturbation at times $t_1$ and $t_1-t_2$ (in red), the hyperbolic-like behaviour of $f$ does the rest of the work for you. In red: the perturbation that we will make during step 7.}\label{FigRotaGlobC1}
\end{center}
\end{figure}

\begin{lemme}\label{LemLyapPerturb}
For every $y'\in F_{x_6}^{g_6}$ such that $d(y',x_6) > d(z',x_6)\nu^{T'\tau_1}$ ($T'$ being defined by Equation~\eqref{hyphyphyp}), there exists a diffeomorphism $g_7$ close to $g_6$ and $T''\in\N$ such that $g_7^{-\tau_1 T''}(z') = y'$. Moreover, the perturbations made to obtain $g_7$ are contained in the linearizing neighbourhood of $\omega_{x_6}$, do not modify the images of $\omega_{x_6}$, nor these of the negative orbit of $z'$ by the discretization or these of the positive orbit of $y'$ in the linearizing neighbourhood of $\omega_{x_6}$
\end{lemme}

\begin{figure}[t]
\begin{center}
\begin{tikzpicture}[scale=1]

\draw[->,>=stealth] (0,0) -- (10.5,0);
\draw (10.5,0) node[right]{$F_{x_6}^{g_6}$};
\draw (.5,0) node {$|$};
\draw (.5,-.5) node {$x_6$};

\draw (1.5,0) node {$\times$};
\draw (1.5,-.5) node {$z'$};
\draw[->,>=stealth,color=blue!50!black] (1.7,-.1) to[bend right] (2.8,-.1);
\draw[color=blue!50!black] (2.25,-.7) node {$g_6^{-\tau_1}$};

\draw[color=blue!50!black] (3,0) node {$\times$};
%\draw (4.5,-.5) node {$g_5^{-\tau_1}(z)$};
\draw[dotted] (3,0) circle (1.2);
\draw[dashed] (3,0) circle (.6);
\draw[->,>=stealth,color=blue!50!black] (3.2,-.1) to[bend right] (5.7,-.1);
\draw[color=blue!50!black] (4.45,-.9) node {$g_6^{-\tau_1}$};

\draw[color=blue!50!black] (5.9,0) node {$\times$};

\draw[->,>=stealth,color=red!50!black] (1.7,.1) to[bend left] (3.4,.1);
\draw[->,>=stealth,color=red!50!black,dotted] (1.7,.1) to[bend left] (2.8,.1);
\draw[->,>=stealth,color=red!50!black] (3.8,.1) to[bend left] (8.2,.1);
\draw[->,>=stealth,color=red!50!black,dotted] (3.8,.1) to[bend left] (7,.1);
\draw[color=red!50!black] (3.6,0) node {$\times$};
\draw[color=red!50!black] (8.4,0) node {$\times$};
\draw[dotted] (7.2,0) circle (2.4);
\draw[dashed] (7.2,0) circle (1.2);

\draw[color=red!50!black] (8.4,0) node {$\times$};

\end{tikzpicture}
\caption[Perturbation made to merge orbits]{Perturbation such that the point $y'_6$ belongs to the negative orbit of $z'$: the initial orbit is drawn in blue (below) and the perturbed orbit in red (above). From a certain time, the red orbit overtakes the blue orbit.}\label{FigLyapPerturb}
\end{center}
\end{figure}
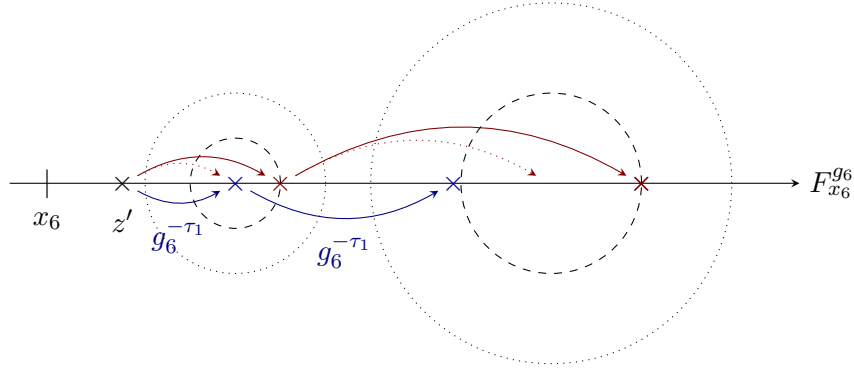

\begin{proof}[Proof of Lemma~\ref{LemLyapPerturb}]
During this proof, if $r$ and $s$ are two points of $W^s(x_6)$, we will denote by $[r,s]$ the segment of $W^s(x_6)$ between $r$ and $s$. Remark that if $r$ and $s$ lie in the neighbourhood of $x_6$ where $g_6$ is linear, then $[r,s]$ is a real segment, included in $F_{x_6}^{g_6}$ moreover, we will denore $[r,+\infty[$ the connected component of $W^s(x_6)\setminus\{r\}$ which does not contain $x_6$.

Consider the point $z'\in F_{x_6}^{g_6}$, and choose a point 
\[p\in \left[g_6^{-\tau_1}(z'), \left(1+\frac{1}{2(1+\eta)}\right)g_6^{-\tau_1}(z')\right].\]
By applying an elementary perturbation (Lemma~\ref{PerturbElem})  whose support is contained into $B\big(g_6^{-\tau_1}(z'),\, d(x_6,g_6^{-\tau_1}(z'))/2\big)$, it is possible to perturb $g_6$ into a diffeomorphism $g_7$ such that $g_6^{\tau_1}(z') = p$ (see Figure~\ref{FigLyapPerturb}). Applying this process $t$ times, for every
\[p\in \left[g_6^{-\tau_1t}(z'), \left(1+\frac{1}{2(1+\eta)}\right)^{t}g_6^{-\tau_1t}(z')\right],\]
it is possible to perturb $g_6$ into a diffeomorphism $g_7$ such that $g_7^{-\tau_1 t}(z') = p$ (the supports of the perturbations are disjoint because the expansion of $g^{-\tau_1}_{|F_{x_6}^{g_6}}$ is bigger than 3). But as $T'$ satisfies Equation~\eqref{hyphyphyp}, the union
\[\bigcup_{t\ge 0}\left[g_6^{-\tau_1t}(z'), \left(1+\frac{1}{2(1+\eta)}\right)^{t}g_6^{-\tau_1t}(z')\right]\]
covers all the interval $[g_6^{-\tau_1T'}(z'), +\infty[$. By the hypothesis made on $y'$, we also have $y'\in [g_6^{-\tau_1T'}(z'), +\infty[$; this proves the lemma.
\end{proof}
Thus, by hypothesis (iv), it is possible to apply Lemma~\ref{LemLyapPerturb} to our setting. This gives us a diffeomorphism $g_7$.

\paragraph{Step 8: final perturbation to put the segment of orbit on the grid.}

To summarize, we have a diffeomorphism $g_7\in\mathcal V'$, and periodic orbit $\omega_{x_7}$ of $g_7$, stabilized by $(g_7)_N$, which bears a measure close to $\mu$. We also have a segment of real orbit of $g_7$ which links the points $y_7\in U$ and $z$, where $z$ is such that $(g_7)_N^{t_4}(z) \in (\omega_{x_7})_N$. To finish the proof of the lemma, it remains to perturb $g_7$ so that the segment of orbit which links the points $y_7$ and $z$ is stabilized by the discretization $(g_7)_N$.
\bigskip

We now observe that by the construction we have made, the distance between two different points of the segment of orbit under $g_7$ between $y_7$ and $z$ is bigger than $2(1+\eta')/N$, and the distance between one point of this segment of orbit and a points of $\omega_{x_7}$ is bigger than $(1+\eta')/N$.

Indeed, if we take one point of the segment of forward orbit $z, g_7^{-1} (z),\cdots, g_7^{-t_2-t_3} (z)$, and one point in the periodic orbit $\omega_{x_7}$, this is due to the hypothesis $\|z - x_7\|\ge R_0/N$ ($R_0$ being defined by Equation~\eqref{DefRr0}) combined with Equation~\eqref{EqPasIdee}. If we take one point in this segment $z, g_7^{-1} (z),\cdots, g_7^{-t_2-t_3} (z)$, and one among the rest of the points (that is, the segment of orbit between $y_7$ and $z$), this is due to the fact that the Lyapunov exponent of $g_7^{\tau_1}$ in $x_7$ is bigger than 3, and to Equation~\eqref{LoinOrbPer}.

If we take one point of the form $g_7^{-t}(z)$, with $t> t_2+t_3$, but belonging to the neighbourhood of $\omega_{x_7}$ where $g_7$ is linear, and one point of $\omega_{x_7}$, this follows from the estimation given by Equation~\eqref{Pastropgros} applied to $\|v\|\ge R_0/N$. If for the second point, instead of considering a point of $\omega_{x_7}$, we take an element of the segment of orbit between $y_7$ and $z$, this follows from the fact that the Lyapunov exponent of $g_7^{\tau_1}$ in $x_7$ is bigger than 3.

Finally, for the points of the orbit that are not in the neighbourhood of $\omega_{x_7}$ where $g_7$ is linear, the property arises from hypothesis (iii) made on $N$.

Thus, by Lemma~\ref{Lem*}, we are able to perturb each of the points of the segment of orbit under $g_7$ between $y_7$ and $z$, such that each of these points belongs to the grid. This gives us a diffeomorphism $g_8\in \mathcal V$. 

To conclude, we have a point $y_8\in U$ whose orbit under $(g_8)_N$ falls on the periodic orbit $(\omega_{x_7})_N$, which bears a measure $\varep$-close to $\mu$. The lemma is proved.
\bigskip

The proof in higher dimensions is almost identical. The perturbation lemmas are still true\footnote{In particular, Lemma~\ref{LocTrans} can be obtained by considering a plane $(P)$ containing both $x$ and $y$ and taking a foliation of $\R^n$ by planes parallel to $(P)$. The desired diffeomorphism is then defined on each leave by the time-$\psi(t)$ of the Hamiltonian given in the proof of the lemma, with $\psi$ is a smooth compactly supported map on the space $\R^n/(P)$, equal to 1 in $0$ and with small $C^1$ norm.}, and the arguments easily adapts by considering the ``super-stable'' manifold of the orbit $\omega_x$, that is the set of points $y\in\T^n$ whose positive orbit is tangent to the Oseledets subspace corresponding to the maximal Lyapunov exponent. In particular, Lemma~\ref{ErgoLemPlus} is still true in this setting, and the connecting lemma for pseudo-orbits {MR2090361} implies that generically, this ``super-stable'' manifold is dense in $\T^n$.
\end{proof}

The proofs of the two statements of the addendum are almost identical.

For the first statement (the fact that for every $\varep>0$, the basin of attraction of the discrete measure can be supposed to contain a $\varep$-dense subset of the torus), we apply exactly the same proof than that of Lemma~\ref{LemMesPhysDiff}: making smaller perturbations of the diffeomorphism if necessary, we can suppose that the stable manifold of $y_8$ is $\varep$-dense. Thus, there exists a segment of backward orbit of $y_8$ which is $\varep$-dense, and we apply the same strategy of proof consisting in putting this segment of orbit on the grid. 

For the second statement, it suffices to apply the strategy of the first statement, and to conjugate the obtained diffeomorphism $g_9$ by an appropriate conservative diffeomorphism with small norm (this norm can be supposed to be as small as desired by taking $\varep$ small), so that the image of the $\varep$-dense subset of $\T^2$ by the conjugation contains the set $E$.
\bigskip

To obtain Theorem~\ref{TheoMesPhysDiffDissip} (dealing with the dissipative case) it suffices to replace the use of Lemmas~\ref{LemAnglesOsel} and~\ref{DerTheoPart2} by the following easier statement.

\begin{lemme}\label{RemplaceDissip}
For every $\alpha>0$ and every $R_0>0$, there exists three times $t_1$, $t_2 \ge 0$ and $t_4\ge 0$ such that:
\begin{itemize}
\item there exists $v\in T \T^n_{f^{t_1}(x)}\cap \Z^n$ such that $\|v\|\ge R_0$ and
\[\big(\widehat{Df}_{f^{t_1 + t_4}(x)}\circ \cdots \circ \widehat{Df}_{f^{t_1}(x)}\big)(v) = 0;\]
\item if $v\in T \T^n_{f^{t_1}(x)}$ is such that the angle between $v$ and $G_{f^{t_1}(x)}$ is bigger than $\alpha$, then the angle between $Df^{-t_2}_{f^{t_1}(x)} v$ and $F_{f^{t_1 - t_2}(x)}$ is smaller than $\alpha$ (see Figure~\ref{FigRotaGlobC1}).
\end{itemize}
\end{lemme}

\begin{proof}[Proof of Lemma~\ref{RemplaceDissip}]
This comes from Oseledets theorem and the hypotheses made on the Lyapunov exponents of $x$, and in particular that their sum is strictly negative.
\end{proof}

\section{Numerical simulations}\label{NumSimPhys}

In this section, we present the results of the numerical simulations we have conducted in connection with Theorem~\ref{TheoMesPhysDiff}.

\subsection[Simulations of the measures $\mu^{f_N}_{x}$]{Simulations of the measures $\mu^{f_N}_{x}$ for conservative torus diffeomorphisms}

We have computed numerically the measures $\mu^{f_N}_x$ for conservative diffeomorphisms $f\in\Diff^1(\T^2,\Leb)$, for the uniform grids
\[E_N = \left\{\left(\frac{i}{N},\frac{j}{N}\right)\in \T^2 \big|\  0\le i,j\le {N}-1\right\},\]
and for starting points $x$ either equal to $(1/2,1/2)$, or chosen at random. We present images of sizes $128\times 128$ pixels representing in logarithmic scale the density of the measures $\mu^{f_N}_x$: each pixel is coloured according to the measure carried by the set of points of $E_N$ it covers. Blue corresponds to a pixel with very small measure and red to a pixel with very high measure. Scales on the right of each image corresponds to the measure of one pixel on the $\log 10$ scale: if green corresponds to $-3$, then a green pixel will have measure $10^{-3}$ for $\mu^{f_N}_x$. For information, when Lebesgue measure is represented, all the pixels have a value about $-4.2$.

We have carried out the simulations on two different diffeomorphisms.\label{DefDiffeoPhys}
\begin{itemize}
\item The first conservative diffeomorphism $f_1$ is of the form $f_1= Q\circ P$, where both $P$ and $Q$ are homeomorphisms of the torus that modify only one coordinate:
\[P(x,y) = \big(x,y+p(x)\big)\quad\text{and}\quad Q(x,y) = \big(x+q(y),y\big),\]
with
\[p(x) = \frac{1}{209}\cos(2\pi\times 17x)+\frac{1}{471}\sin(2\pi\times 29x)-\frac{1}{703}\cos(2\pi\times 39x),\]
\[q(y) = \frac{1}{287}\cos(2\pi\times 15y)+\frac{1}{403}\sin(2\pi\times 31y)-\frac{1}{841}\sin(2\pi\times 41y).\]
This $C^\infty$-diffeomorphism is in fact $C^1$-close to the identity. This allows $f_1$ to admit periodic orbits with not too large periods. Note that $f_1$ is also chosen so that it is not $C^2$-close to the identity. 

\item The second conservative diffeomorphism $f_2$ is the composition $f_2 = f_1\circ A$, with $A$ the linear Anosov map
\[A = \begin{pmatrix} 2 & 1 \\ 1 & 1 \end{pmatrix}.\]
As $f_1$ is $C^1$-close to $\Id$, the diffeomorphism $f_2$ is $C^0$-conjugated to the linear automorphism $A$, which is in particular ergodic.
\end{itemize}
\bigskip

To compute these measures, we used Floyd's Algorithm (or the ``tortoise and the hare algorithm''). It has appeared that on the examples of diffeomorphisms we have tested, we were able to test orders of discretization $N\simeq 2^{20}$. Thus, the first figures represent the measures $\mu^{f_N}_x$ for $N\in\llbracket 2^{20}+1,2^{20}+9\rrbracket$. We have also computed the distance between the measure $\mu^{f_N}_x$ and Lebesgue measure (see Figure~\ref{GrafDistLebPhys}). The distance we have chosen is given by the formula
\[d(\mu,\nu) = \sum_{k=0}^\infty \frac{1}{2^k} \sum_{i,j=0}^{2^k-1} \big| \mu(C_{i,j,k}) - \nu(C_{i,j,k})\big|\in[0,2],\]
where
\[C_{i,j,k} = \left[\frac{i}{2^k},\frac{i+1}{2^k}\right] \times \left[\frac{j}{2^k},\frac{j+1}{2^k}\right].\]
In practice, we have computed an approximation of this quantity by summing only on the $k\in\llbracket 0,7 \rrbracket$.

\begin{figure}[h]
\begin{center}
\makebox[0.8\textwidth]{\parbox{0.8\textwidth}{%
\begin{minipage}[c]{.49\linewidth}
	\includegraphics[width=\linewidth,trim = .5cm .3cm .6cm .1cm,clip]{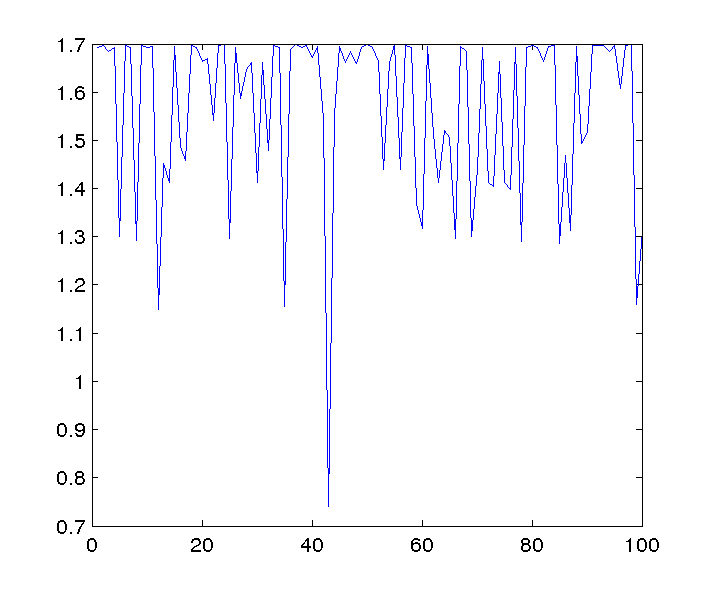}
\end{minipage}\hfill
\begin{minipage}[c]{.49\linewidth}
	\includegraphics[width=\linewidth,trim = .5cm .3cm .6cm .1cm,clip]{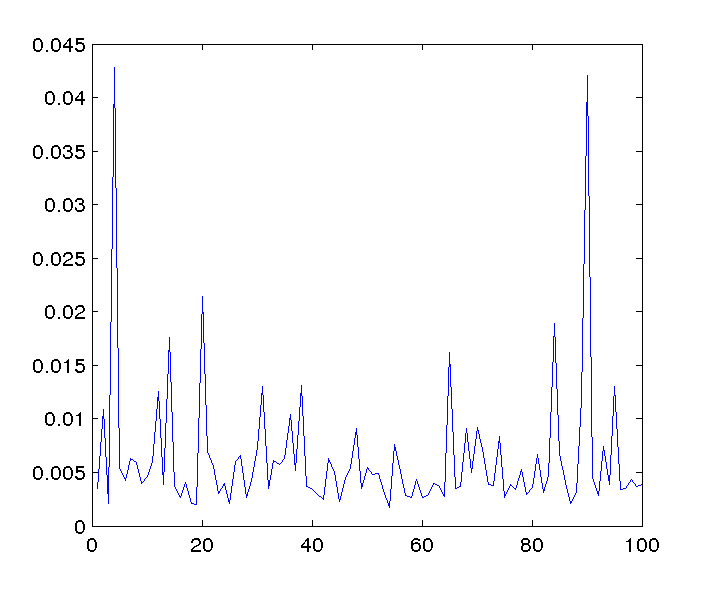}
\end{minipage}
}}
\end{center}
\caption[Simulation of the distance between $\Leb$ and $\mu^{(f_i)_N}_{(1/2,1/2)}$ for 2 examples of conservative diffeomorphisms]{Distance between Lebesgue measure and the measure $\mu^{(f_i)_N}_{(1/2,1/2)}$ depending on $N$ for $f_1$ (left) and $f_2$ (right), on the grids $E_N$ with $N=2^{20}+k$, $k=1,\cdots,100$.}\label{GrafDistLebPhys}
\end{figure}

\bigskip

\begin{figure}[ht]
\begin{minipage}[c]{.31\linewidth}
	\includegraphics[height=4.8cm,trim = 1.5cm .95cm 2.8cm .5cm,clip]{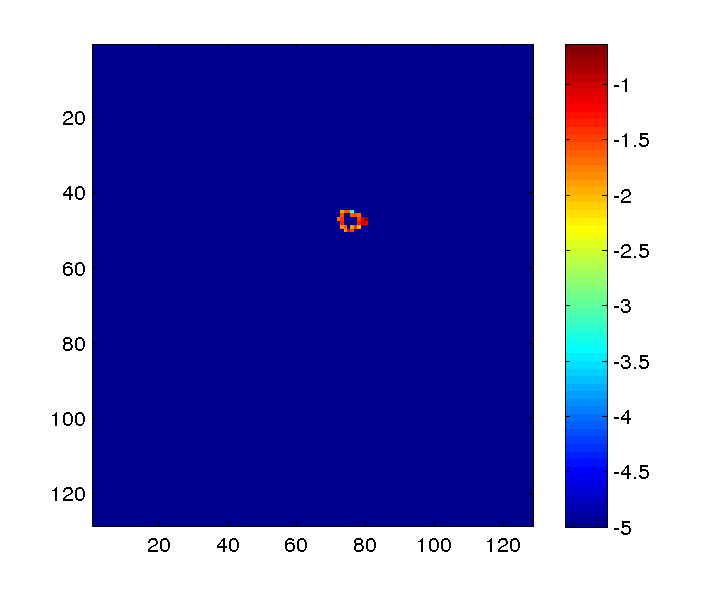}
\end{minipage}\hfill
\begin{minipage}[c]{.31\linewidth}
	\includegraphics[height=4.8cm,trim = 1.5cm .95cm 2.8cm .5cm,clip]{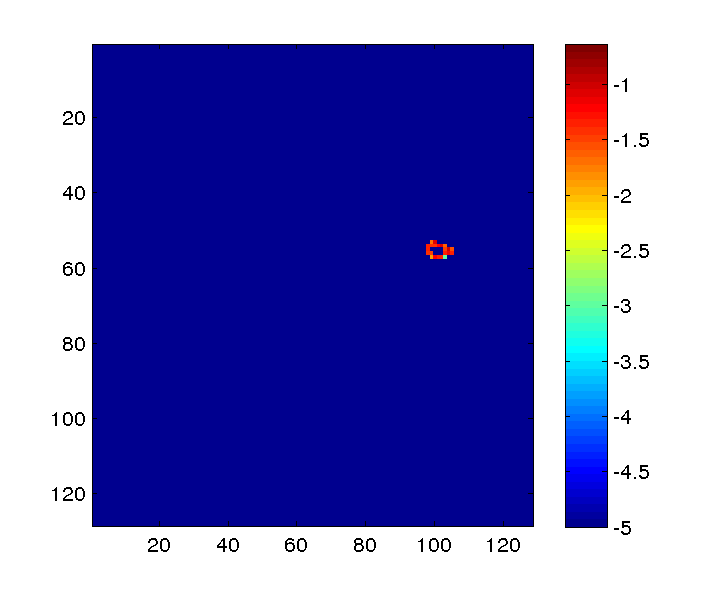}
\end{minipage}\hfill
\begin{minipage}[c]{.37\linewidth}
	\includegraphics[height=4.8cm,trim = 1.5cm .95cm 1cm .5cm,clip]{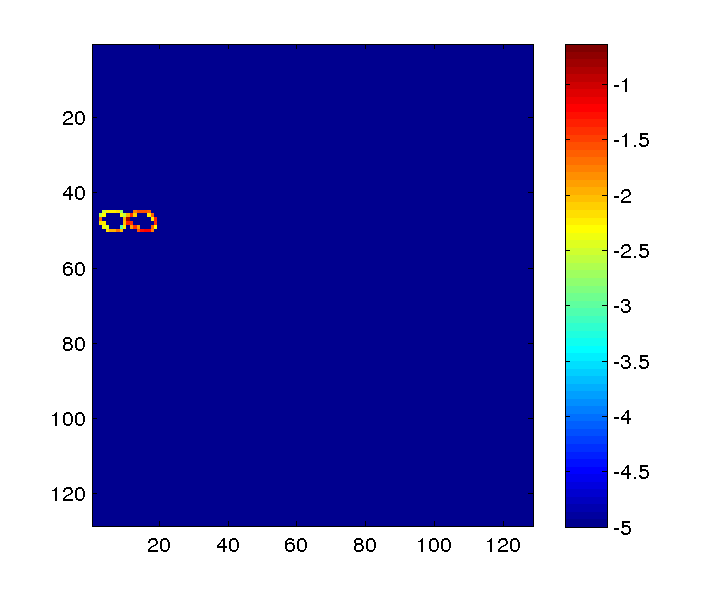}
\end{minipage}

\begin{minipage}[c]{.31\linewidth}
	\includegraphics[height=4.8cm,trim = 1.5cm .95cm 2.8cm .5cm,clip]{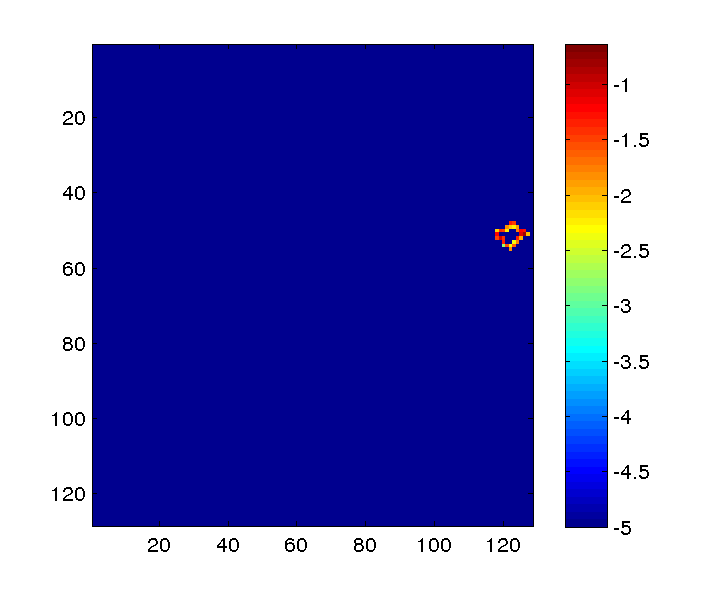}
\end{minipage}\hfill
\begin{minipage}[c]{.31\linewidth}
	\includegraphics[height=4.8cm,trim = 1.5cm .95cm 2.8cm .5cm,clip]{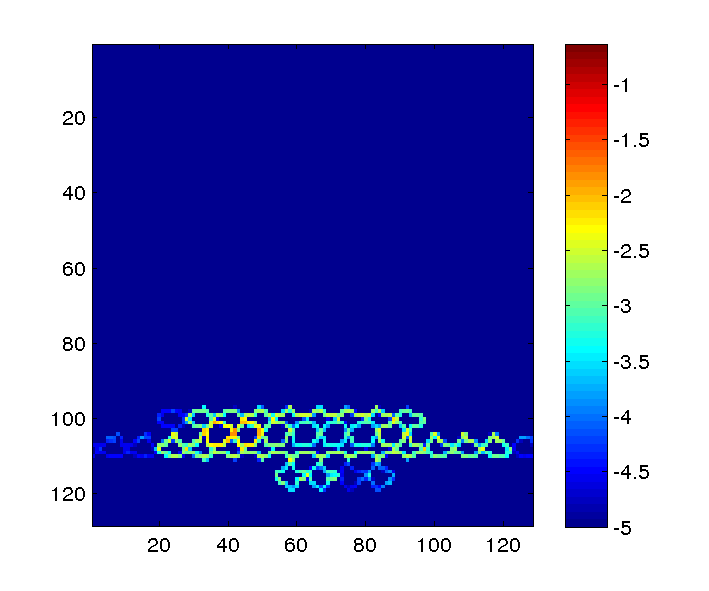}
\end{minipage}\hfill
\begin{minipage}[c]{.37\linewidth}
	\includegraphics[height=4.8cm,trim = 1.5cm .95cm 1cm .5cm,clip]{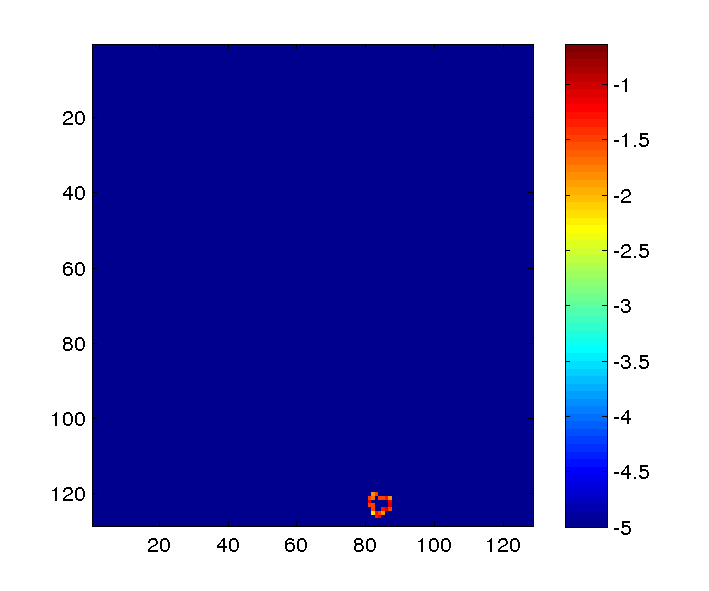}
\end{minipage}

\caption[Simulations of $\mu^{(f_1)_N}_x$ on the grids $E_N$, with $N=2^{20}+i$, $i=1,\cdots,9$]{Simulations of invariant measures $\mu^{(f_1)_N}_x$ on the grids $E_N$, with $N=2^{20}+i$, $i=1,\cdots,6$ and $x=(1/2,1/2)$ (from left to right and top to bottom).}\label{MesPhysIdC1}
\end{figure}

\begin{figure}[ht]
\begin{minipage}[c]{.31\linewidth}
	\includegraphics[height=4.8cm,trim = 1.5cm .95cm 2.8cm .5cm,clip]{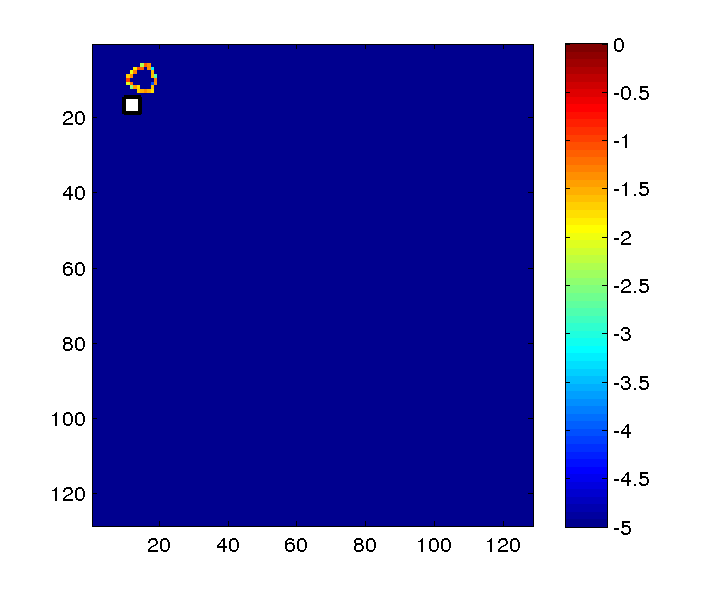}
\end{minipage}\hfill
\begin{minipage}[c]{.31\linewidth}
	\includegraphics[height=4.8cm,trim = 1.5cm .95cm 2.8cm .5cm,clip]{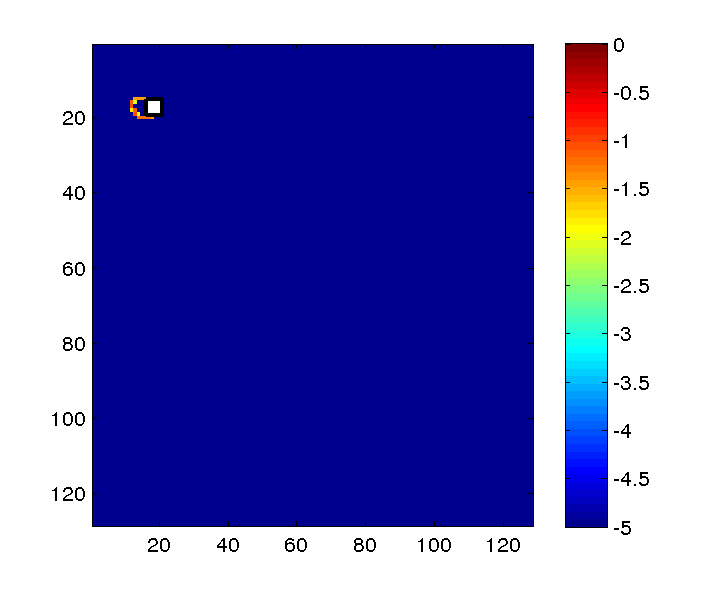}
\end{minipage}\hfill
\begin{minipage}[c]{.37\linewidth}
	\includegraphics[height=4.8cm,trim = 1.5cm .95cm 1cm .5cm,clip]{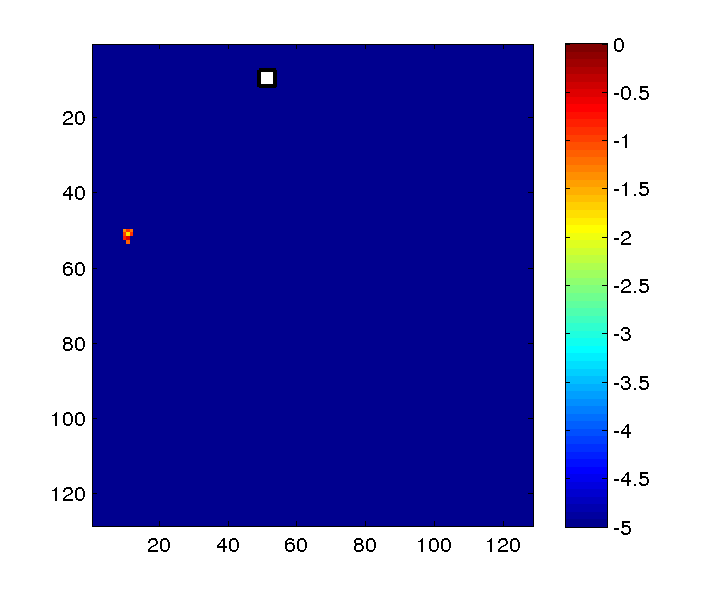}
\end{minipage}

\begin{minipage}[c]{.31\linewidth}
	\includegraphics[height=4.8cm,trim = 1.5cm .95cm 2.8cm .5cm,clip]{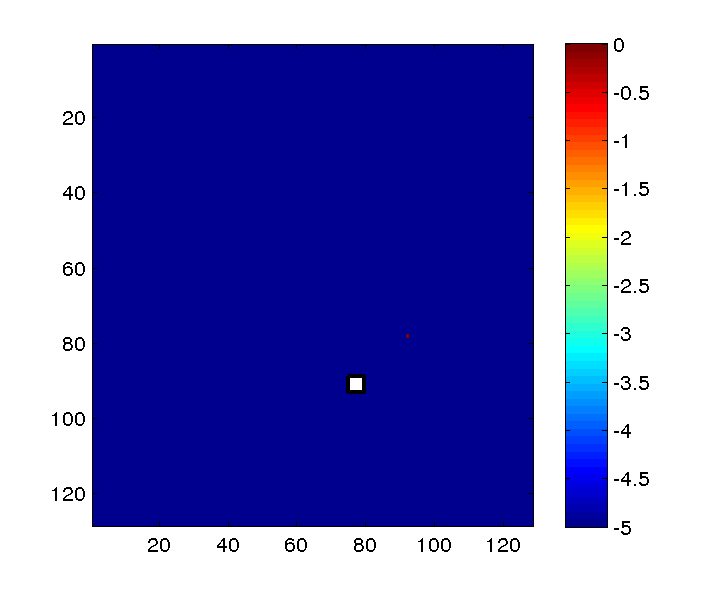}
\end{minipage}\hfill
\begin{minipage}[c]{.31\linewidth}
	\includegraphics[height=4.8cm,trim = 1.5cm .95cm 2.8cm .5cm,clip]{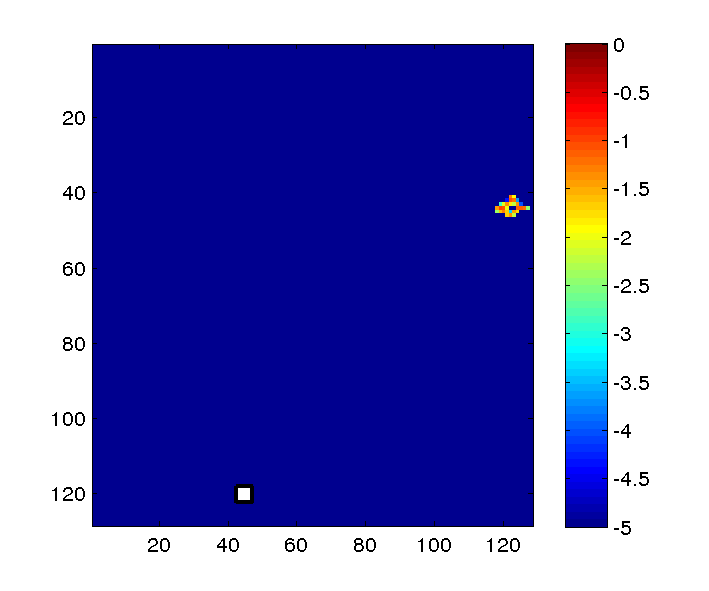}
\end{minipage}\hfill
\begin{minipage}[c]{.37\linewidth}
	\includegraphics[height=4.8cm,trim = 1.5cm .95cm 1cm .5cm,clip]{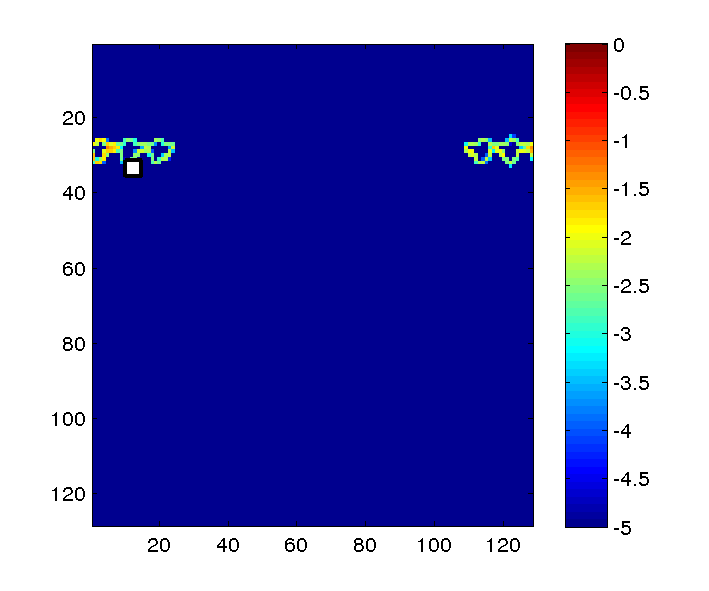}
\end{minipage}
\caption[Simulations of $\mu^{(f_1)_N}_x$ on the grid $E_N$, with $N=2^{23}$ and a random point $x$]{Simulations of invariant measures $\mu^{(f_1)_N}_x$ on the grid $E_N$, with $N=2^{23}$, and $x$ a random point of $\T^2$, represented by the black and white box. The behaviour observed on the top left picture is the most frequent, but we also observe other kind of measures: for example, the measures has a very small support like on the bottom left picture on about 10 of the 100 random draws we have made; we even see appearing the strange behaviour of the last picture once.}\label{MesPhysIdC1Pt}
\end{figure}

In the case of the diffeomorphism $f_1$, which is close to the identity, we observe a strong variation of the measure $\mu^{(f_1)_N}_x$ depending on $N$ (left of Figure~\ref{GrafDistLebPhys} and Figure~\ref{MesPhysIdC1}). More precisely, for 7 on the 9 orders of discretization represented on Figure~\ref{MesPhysIdC1}, these measures seem to be supported by a small curve; for $N = 2^{20}+3$, this measure seems to be supported by a figure-8 curve, and for $N = 2^{20}+5$, the support of the measure is quite complicated and looks like an interlaced curve. The fact that the measures $\mu^{(f_1)_N}_x$ strongly depend on $N$ reflects the behaviour predicted by Theorem~\ref{TheoMesPhysDiff}: in theory, for a generic $C^1$ diffeomorphism, the measures $\mu^{f_N}_x$ should accumulate on the whole set of $f$-invariant measures; here we see that these measures strongly depend on $N$ (moreover, we can see on Figure~\ref{GrafDistLebPhys} that on the orders of discretization we have tested, these measures are never close to Lebesgue measure). We have no satisfying explanation to the specific shape of the supports of the measures. When we fix the order of discretization and make vary the starting point $x$, the behaviour is very similar: the measures $\mu^{(f_1)_N}_x$ widely depend on the point $x$ (see Figure~\ref{MesPhysIdC1Pt}). We also remark that increasing the order of discretizations does not make the measures $\mu^{(f_1)_N}_x$ evolve more smoothly.
\bigskip
\bigskip

\begin{figure}[ht]
\begin{minipage}[c]{.31\linewidth}
	\includegraphics[height=4.8cm,trim = 1.5cm .95cm 2.8cm .5cm,clip]{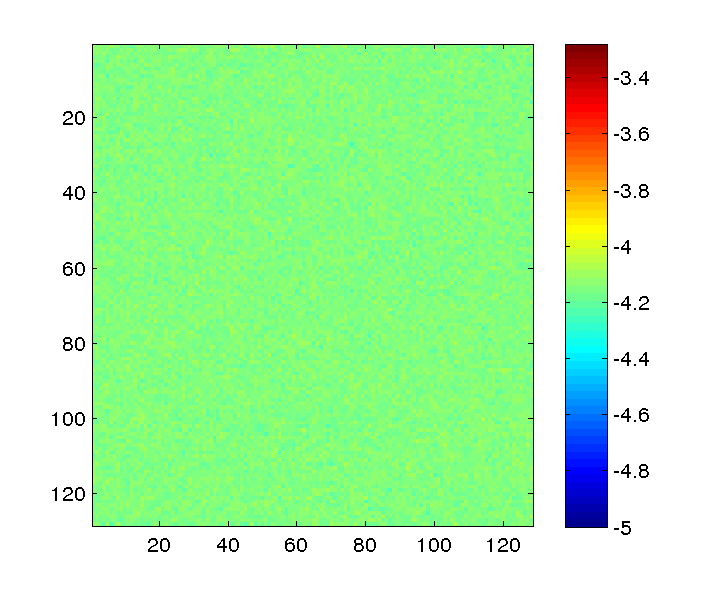}
\end{minipage}\hfill
\begin{minipage}[c]{.31\linewidth}
	\includegraphics[height=4.8cm,trim = 1.5cm .95cm 2.8cm .5cm,clip]{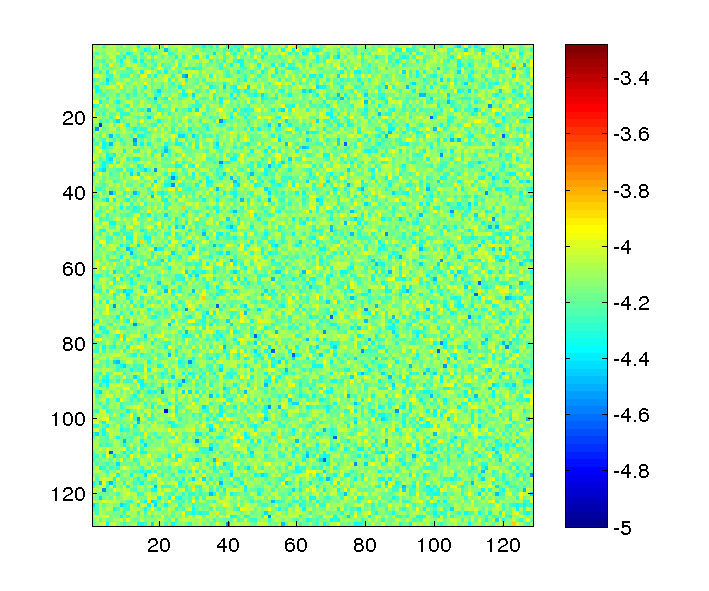}
\end{minipage}\hfill
\begin{minipage}[c]{.37\linewidth}
	\includegraphics[height=4.8cm,trim = 1.5cm .95cm 1cm .5cm,clip]{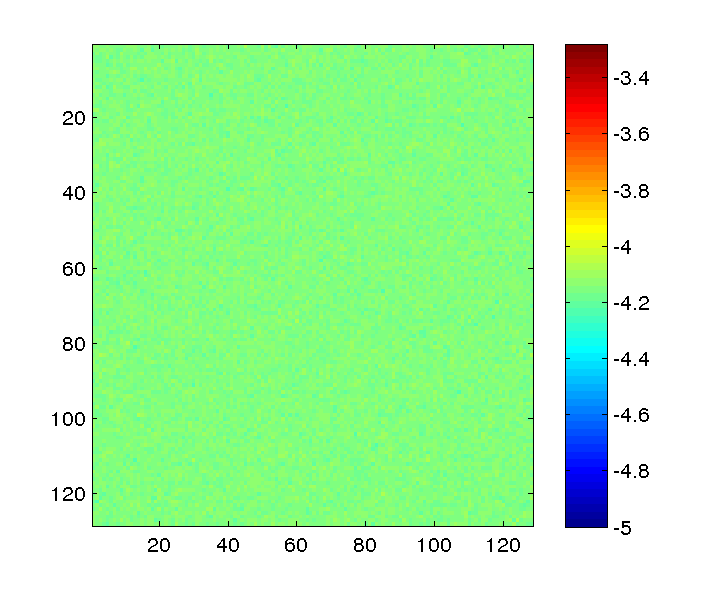}
\end{minipage}

\begin{minipage}[c]{.31\linewidth}
	\includegraphics[height=4.8cm,trim = 1.5cm .95cm 2.8cm .5cm,clip]{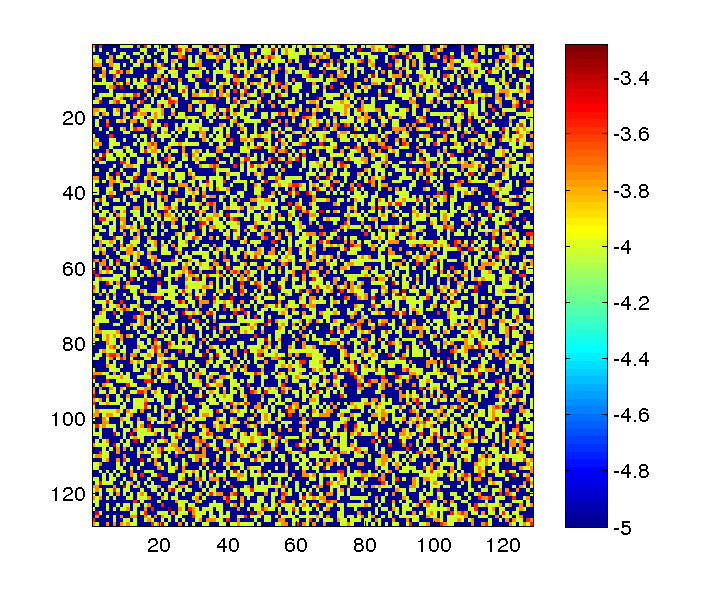}
\end{minipage}\hfill
\begin{minipage}[c]{.31\linewidth}
	\includegraphics[height=4.8cm,trim = 1.5cm .95cm 2.8cm .5cm,clip]{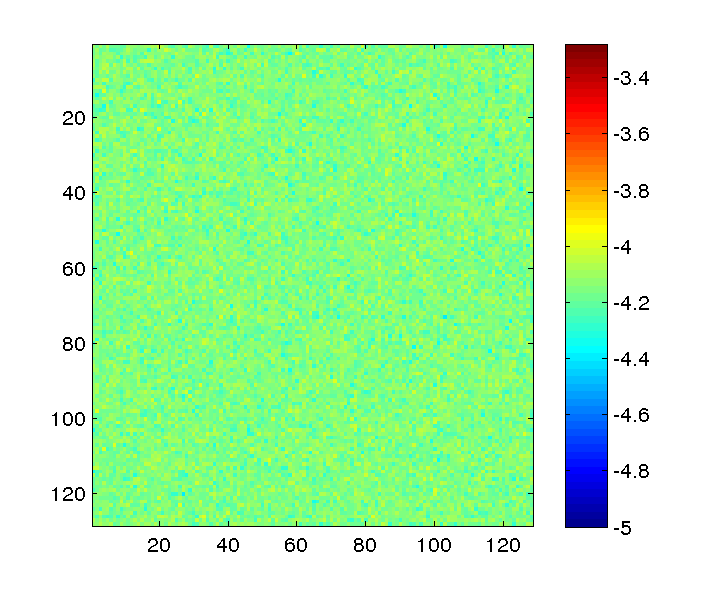}
\end{minipage}\hfill
\begin{minipage}[c]{.37\linewidth}
	\includegraphics[height=4.8cm,trim = 1.5cm .95cm 1cm .5cm,clip]{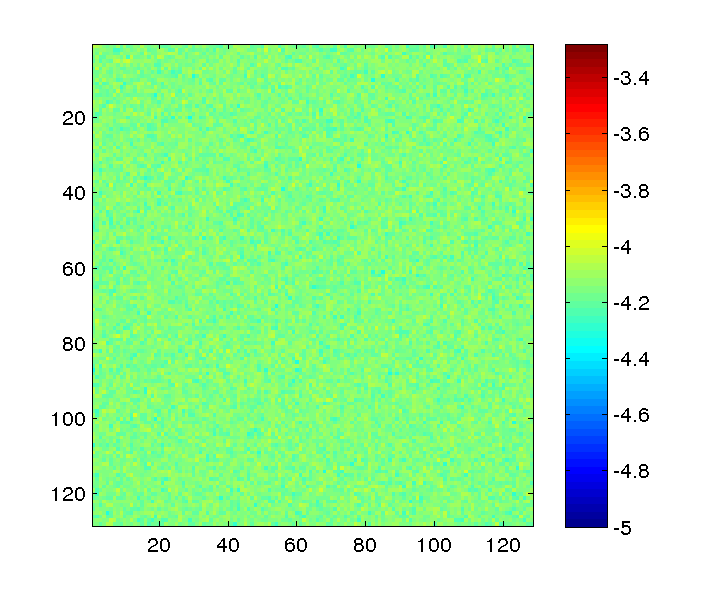}
\end{minipage}

\caption[Simulations of $\mu^{(f_2)_N}_x$ on the grids $E_N$, with $N=2^{20}+i$, $i=1,\cdots,6$]{Simulations of invariant measures $\mu^{(f_2)_N}_x$ on the grids $E_N$,  with $N=2^{20}+i$, $i=1,\cdots,6$ and $x=(1/2,1/2)$ (from left to right and top to bottom).}\label{MesPhysAnoC1}
\end{figure}

\begin{figure}[ht]
\begin{minipage}[c]{.31\linewidth}
	\includegraphics[height=4.8cm,trim = 1.5cm .95cm 2.8cm .5cm,clip]{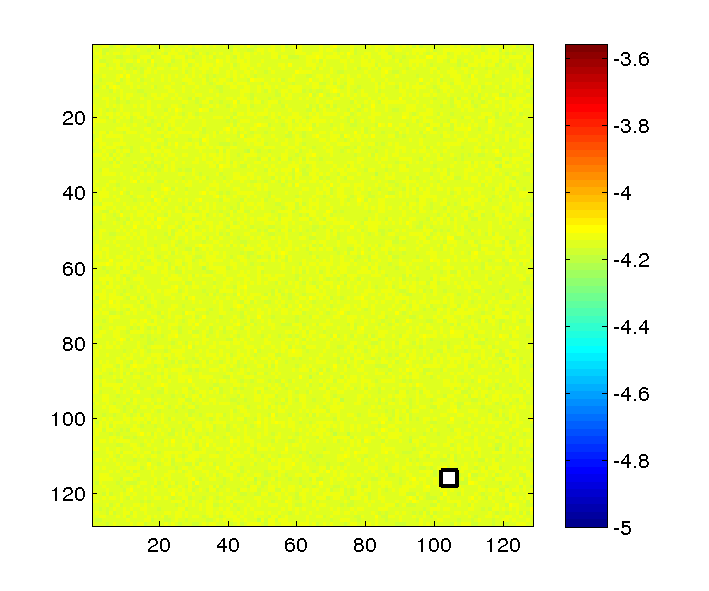}
\end{minipage}
\begin{minipage}[c]{.31\linewidth}
	\includegraphics[height=4.8cm,trim = 1.5cm .95cm 2.8cm .5cm,clip]{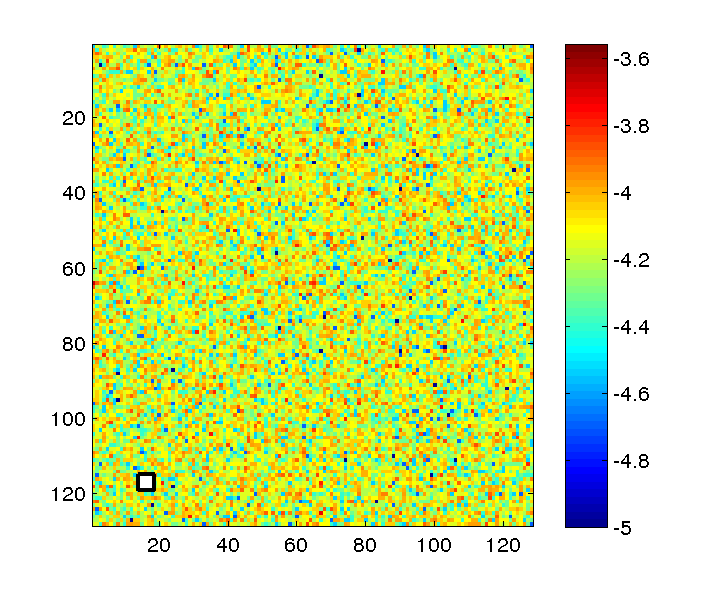}
\end{minipage}\hfill
\begin{minipage}[c]{.37\linewidth}
	\includegraphics[height=4.8cm,trim = 1.5cm .95cm 1cm .5cm,clip]{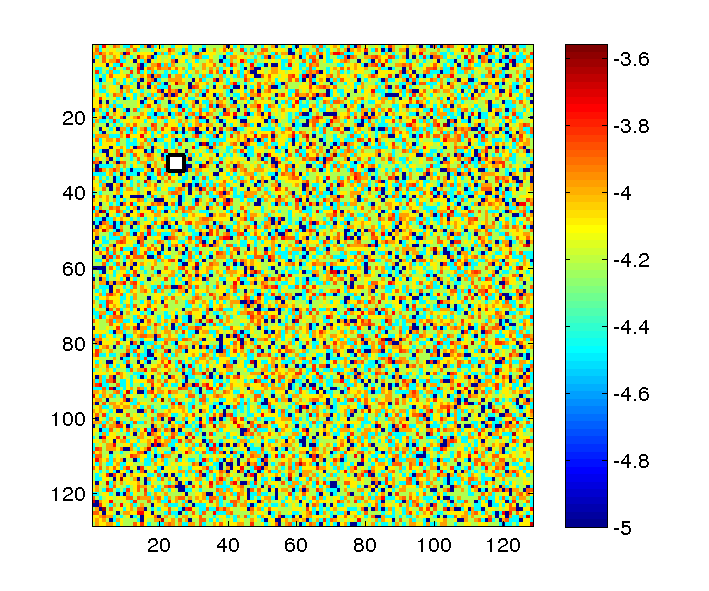}
\end{minipage}

\caption[Simulations of $\mu^{(f_2)_N}_x$ on the grid $E_N$, with $N\simeq 2^{23}$ and a random point $x$]{Simulations of invariant measures $\mu^{(f_2	)_N}_x$ on the grid $E_N$, with $N=2^{23}+1$ (left) and $N=2^{23}+17$ (middle and right), and $x$ a random point of $\T^2$, represented by the black and white box. The behaviour observed on the left picture is the most frequent (for 17 over the 20 orders $N=2^{23}+i$, $i=0,\cdots,19$, all the 100 random draws we have made gave a measure very close to $\Leb$), but seldom we also observe measures further from Lebesgue measure, like what happens for $N=2^{23}+17$ (middle and right), where $99$ over the $100$ random draws of $x$ produce a measure identical to the pictures on the middle, and the other random draw gives a measure a bit more singular with respect to Lebesgue measure (right).}\label{MesPhysAnoC1Pt}
\end{figure}

For $f_2$ (a small $C^1$-perturbation of the linear Anosov map $A$), most of the time, the measures $\mu^{(f_2)_N}_x$ are close to Lebesgue measure, but for one order of discretization $N$ (here, $N = 2^{20}+4$), the measure becomes very different from Lebesgue measure (we can see on the right of Figure~\ref{GrafDistLebPhys} that this phenomenon appears twice when $N\in\llbracket 2^{20}+1,2^{20}+100\rrbracket$). The same phenomenon holds when we fix the order of discretization but change the starting point $x$ (see Figure~\ref{MesPhysAnoC1Pt}), except that the number of apparition of measures that are singular with respect to Lebesgue measure is smaller than in Figure~\ref{MesPhysAnoC1}. Again, we think that this follows from the fact that the orders of discretizations tested are bigger. In this case, the simulations suggest the following behaviour: when the order of discretization $N$ increases, the frequency of apparition of measures $\mu^{(f_2)_N}_x$ far away from Lebesgue measures tends to $0$.

Recall that Addendum~\ref{AddTheoMesPhysDiff} states that if $x$ is fixed, then for a generic $f\in\Diff^1(\T^2,\Leb)$, the measures $\mu_x^{f_N}$ accumulate on the whole set of $f$-invariant measures, but do not say anything about, for instance, the frequency of orders $N$ such that $\mu_x^{f_N}$ is not close to Lebesgue measure. It is natural to think that this frequency depends a lot on $f$; for example that such $N$ are very rare close to an Anosov diffeomorphism and more frequent close to an ``elliptic'' dynamics like the identity. The results of numerical simulations seem to confirm this heuristic.

\subsection[Simulations of the measures $\mu^{f_N}_{\T^2}$]{Simulations of the measures $\mu^{f_N}_{\T^2}$ for conservative torus diffeomorphisms}

We now present the results of numerical simulations of the measures $\mu_{\T^2}^{f_N}$. Recall that the measure $\mu_{\T^2}^{f_N}$ is supported by the union of periodic orbits of $f_N$, and is such that the total measure of each periodic orbit is equal to the cardinality of its basin of attraction.

First, we simulate a conservative diffeomorphism $g_1$ which is close to the identity in the $C^1$ topology. We have chosen $g_1 = Q\circ P$, where
\[P(x,y) = \big(x,y+p(x)\big)\quad\text{and}\quad Q(x,y) = \big(x+q(y),y\big),\]
with
\[p(x) = \frac{1}{209}\cos(2\pi\times 17x)+\frac{1}{271}\sin(2\pi\times 27x)-\frac{1}{703}\cos(2\pi\times 35x),\]
\[q(y) = \frac{1}{287}\cos(2\pi\times 15y)+\frac{1}{203}\sin(2\pi\times 27y)-\frac{1}{841}\sin(2\pi\times 38y).\]

We have also simulated the conservative diffeomorphism $g_2 = g_1\circ A$, with $A$ the standard Anosov automorphism
\[A = \begin{pmatrix} 2 & 1 \\ 1 & 1 \end{pmatrix},\]
thus $g_2$ is a small $C^1$ perturbation of $A$; in particular the theory asserts that it is topologically conjugated to $A$. We can test whether this property can be observed on simulations or not.
\bigskip

\begin{figure}[h!]
\begin{center}
\makebox[0.8\textwidth]{\parbox{0.8\textwidth}{%
\begin{minipage}[c]{.49\linewidth}
	\includegraphics[width=\linewidth,trim = .5cm .3cm .6cm .1cm,clip]{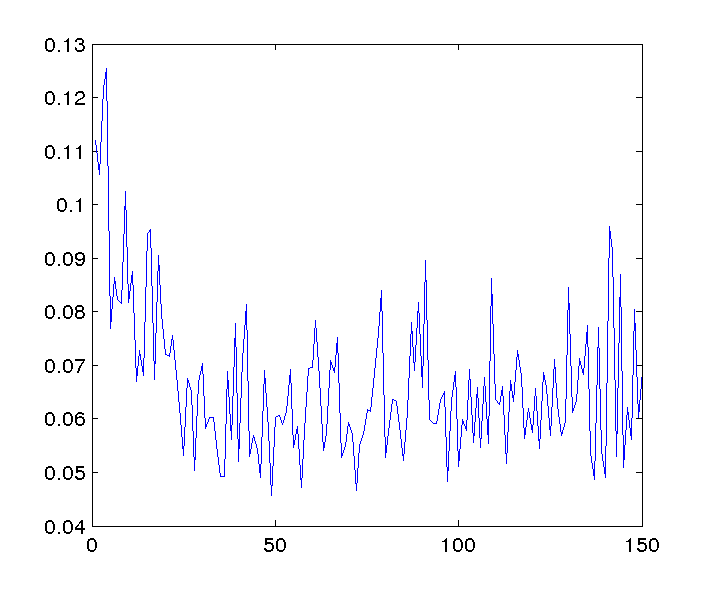}
\end{minipage}\hfill
\begin{minipage}[c]{.49\linewidth}
	\includegraphics[width=\linewidth,trim = .5cm .3cm .6cm .1cm,clip]{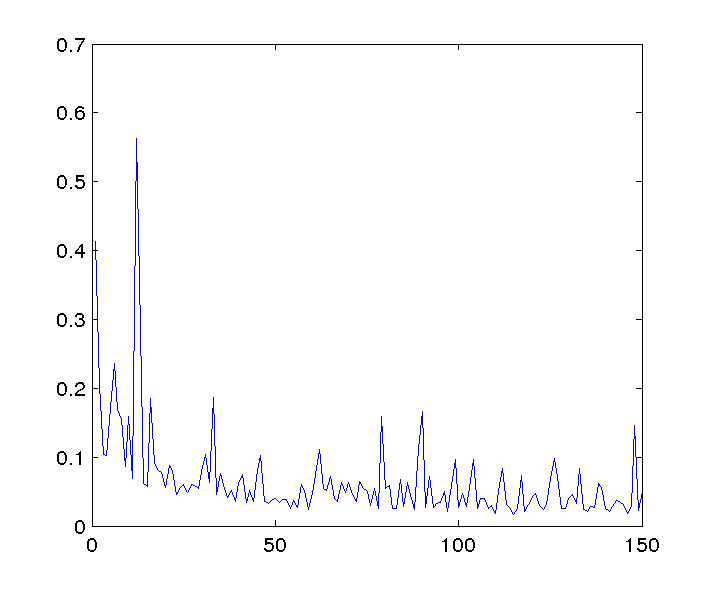}
\end{minipage}
}}\end{center}
\caption[Simulation of the distance between $\Leb$ and $\mu^{(g_i)_N}_{\T^2}$ for 2 examples of conservative diffeomorphisms]{Distance between Lebesgue measure and the measure $\mu^{(g_i)_N}_{\T^2}$ depending on $N$ for $g_1$ (left) and $g_2$ (right), on the grids $E_N$ with $N=128k$, $k=1,\cdots,150$.}\label{GrafDistLebC1Cons}
\end{figure}

For $g_1$, the distance $d(\mu^{f_N}_{\T^2},\Leb)$ is quite quickly smaller than $0.1$, and oscillates between $0.05$ and $0.1$ from $N=128\times 30$. It is not clear if in this case, the sequence of measures $(\mu^{f_N}_{\T^2})_N$ converge towards Lebesgue measure or not (while for the $C^0$ perturbation of the identity, it is clear that these measures do not converge to anything, see \cite{Guih-discr}). The distance $d(\mu^{f_N}_{\T^2},\Leb)$ even seem to increase slowly (in average --- there are a lot of oscillations) from $N=50\times 128$. We have the same kind of conclusion for $g_2$: by looking at Figure~\ref{GrafDistLebC1Cons}, we can not decide if the sequence of measures $(\mu^{f_N}_{\T^2})$ seem to tend to Lebesgue measure or not.
\bigskip

\begin{figure}[ht]
\begin{minipage}[c]{.31\linewidth}
	\includegraphics[height=4.8cm,trim = 1.5cm .95cm 2.8cm .5cm,clip]{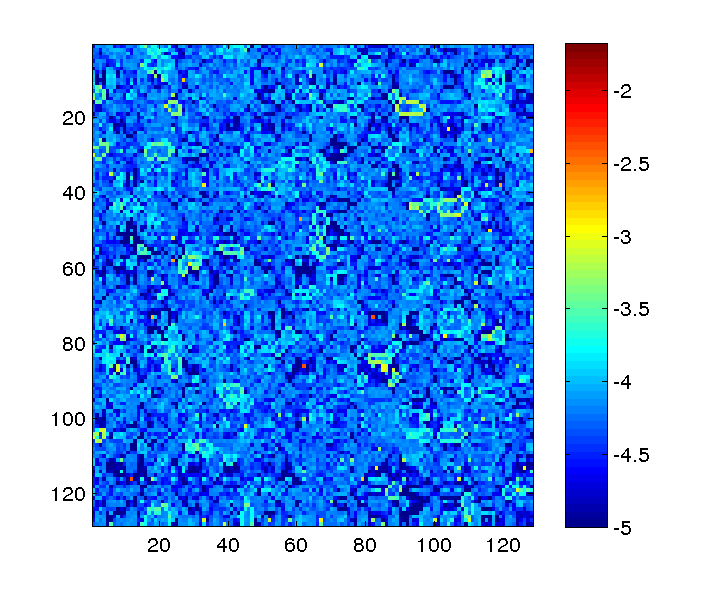}
\end{minipage}\hfill
\begin{minipage}[c]{.31\linewidth}
	\includegraphics[height=4.8cm,trim = 1.5cm .95cm 2.8cm .5cm,clip]{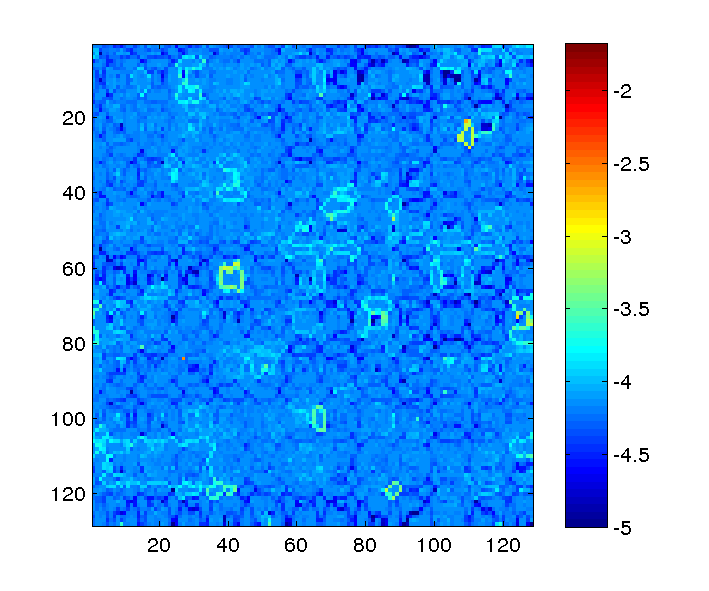}
\end{minipage}\hfill
\begin{minipage}[c]{.37\linewidth}
	\includegraphics[height=4.8cm,trim = 1.5cm .95cm 1cm .5cm,clip]{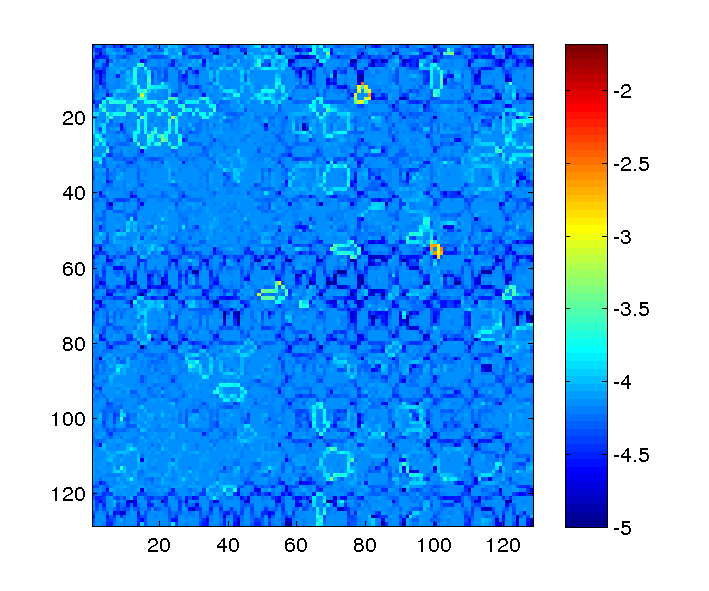}
\end{minipage}
\caption[Simulations of $\mu^{(g_1)_N}_{\T^2}$ on the grids $E_N$, with $N=2^k$, $k= 10,14,15$]{Simulations of invariant measures $\mu^{(g_1)_N}_{\T^2}$ on the grids $E_N$, with $N=2^k$, $k= 10,14,15$ (from left to right and top to bottom).}\label{MesC1IdCons2p}
\end{figure}

The behaviour of the computed invariant measures $\mu^{(g_1)_N}_{\T^2}$, where $g_1$ is a small $C^1$ perturbation of the identity, is way smoother than in the $C^0$ case (compare Figure~\ref{MesC1IdCons2p} with \cite{Guih-discr}). Indeed, the measure $\mu^{(g_1)_N}_{\T^2}$ has quickly a big component which is close to Lebesgue measure: the images contain a lot of light blue. Thus, we could be tempted to conclude that these measures converge to Lebesgue measure. However, there are still little regions that have a big measure: in the example of Figure~\ref{MesC1IdCons2p}, it seems that there are invariant curves that attract a lot of the points of the grid (as can also be observed on Figure~\ref{MesPhysIdC1}). We have no explanation to this phenomenon, and we do not know if it still occurs when the order of discretization is very large.
\bigskip

\begin{figure}[ht]
\begin{minipage}[c]{.31\linewidth}
	\includegraphics[height=4.8cm,trim = 1.5cm .95cm 2.8cm .5cm,clip]{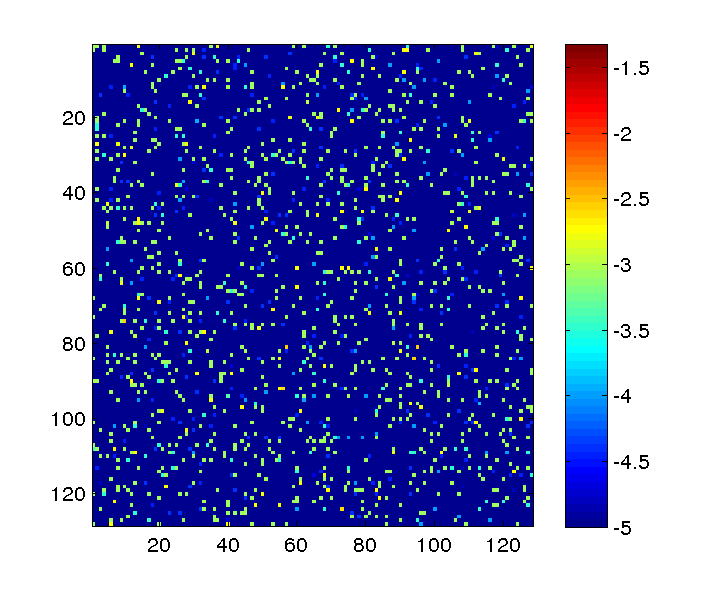}
\end{minipage}\hfill
\begin{minipage}[c]{.31\linewidth}
	\includegraphics[height=4.8cm,trim = 1.5cm .95cm 2.8cm .5cm,clip]{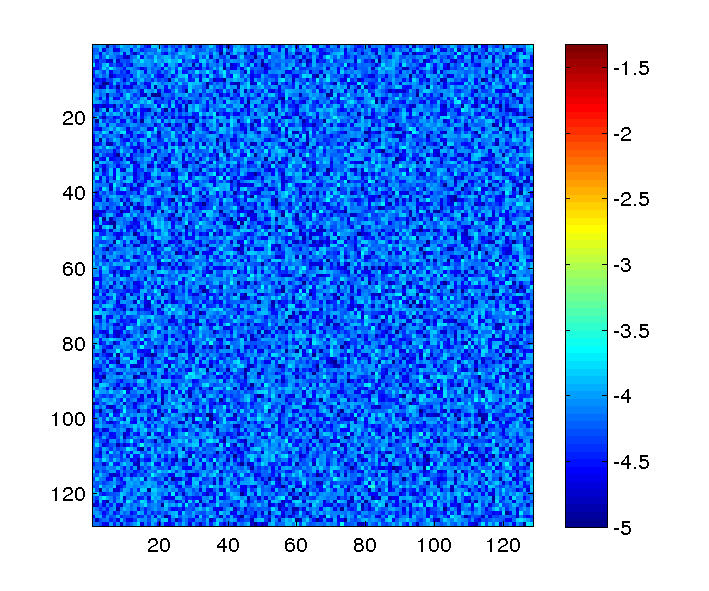}
\end{minipage}\hfill
\begin{minipage}[c]{.37\linewidth}
	\includegraphics[height=4.8cm,trim = 1.5cm .95cm 1cm .5cm,clip]{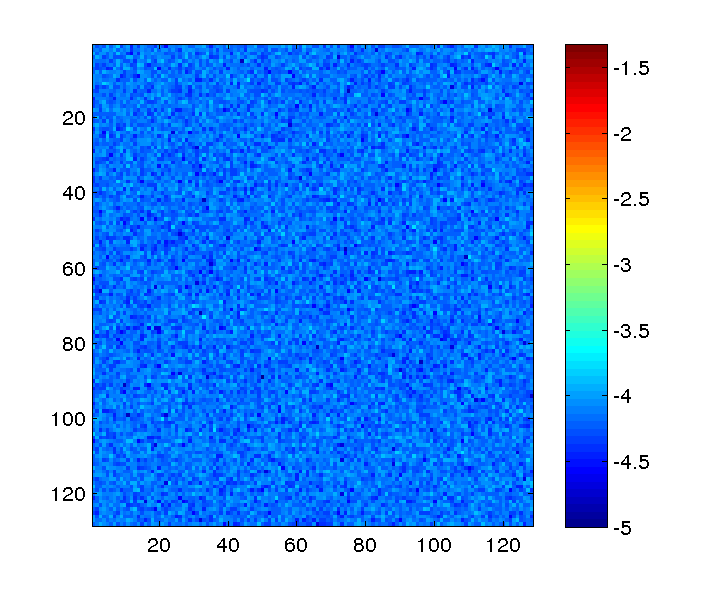}
\end{minipage}
\caption[Simulations of $\mu^{(g_2)_N}_{\T^2}$ on the grids $E_N$, with $N=2^k$, $k= 10,14,,15$]{Simulations of invariant measures $\mu^{(g_2)_N}_{\T^2}$ on the grids $E_N$, with $N=2^k$, $k= 10,14,15$ (from left to right and top to bottom).}\label{MesC1AnoCons2p}
\end{figure}

\begin{figure}[ht]
\begin{minipage}[c]{.31\linewidth}
	\includegraphics[height=4.8cm,trim = 1.5cm .95cm 2.8cm .5cm,clip]{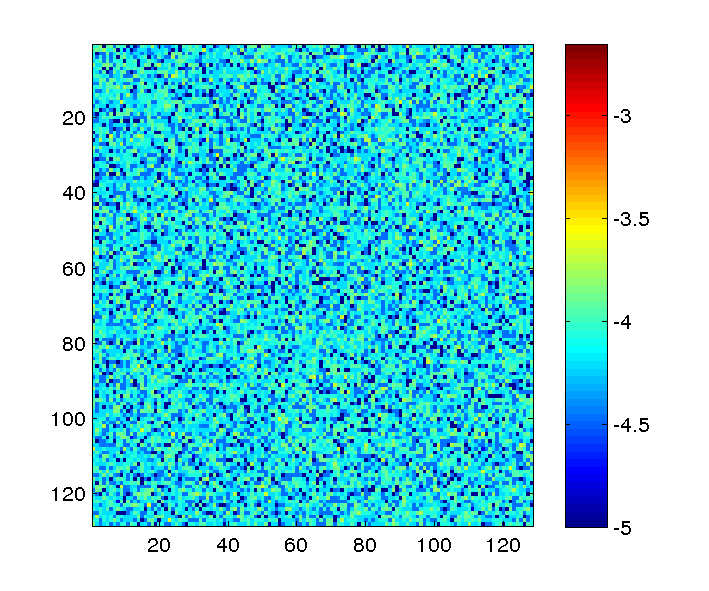}
\end{minipage}\hfill
\begin{minipage}[c]{.31\linewidth}
	\includegraphics[height=4.8cm,trim = 1.5cm .95cm 2.8cm .5cm,clip]{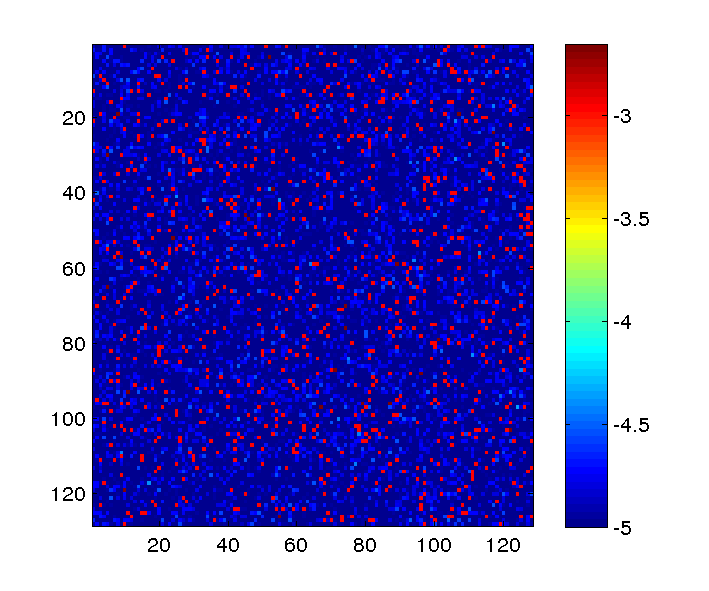}
\end{minipage}\hfill
\begin{minipage}[c]{.37\linewidth}
	\includegraphics[height=4.8cm,trim = 1.5cm .95cm 1cm .5cm,clip]{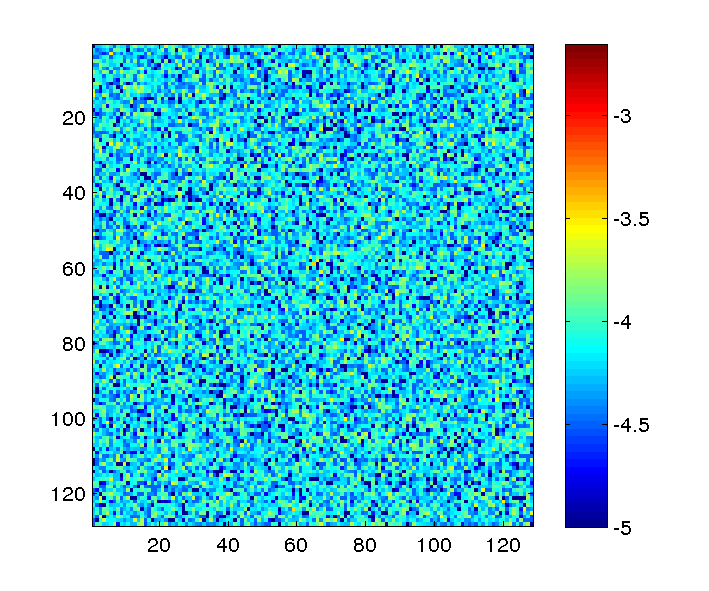}
\end{minipage}

\caption[Simulations of $\mu^{(g_2)_N}_{\T^2}$ on the grids $E_N$, with $N=11\,519, \cdots, 11\,521$]{Simulations of invariant measures $\mu^{(g_2)_N}_{\T^2}$ on the grids $E_N$, with $N=11\,516, \cdots, 11\,524$ (from left to right and top to bottom).}\label{MesC1AnoConsSerie}
\end{figure}

For the discretizations of $g_2$, the simulations on grids of size $2^k\times 2^k$ might suggest that the measures $\mu^{(g_2)_N}_{\T^2}$ tend to Lebesgue measure (Figure~\ref{MesC1AnoCons2p}). Actually, when we perform a lot of simulations, we realize that there are also big variations of the behaviour of the measures (Figure \ref{MesC1AnoConsSerie}): the measure is often well distributed in the torus, and sometimes quite singular with respect to Lebesgue measure (as it can be seen in Figure \ref{GrafDistLebC1Cons}). This behaviour is almost identical to that observed in the $C^0$ case in the neighbourhood of $A$ (see \cite{Guih-discr}).

\appendix

\section{A more general setting where the theorems are still true}\label{AddendSett}

Here, we give weaker assumptions under which the theorems of this paper are still true: the framework ``torus $\T^n$ with grids $E_N$ and Lebesgue measure'' could be seen as a little too restrictive.
\bigskip

So, we take a compact smooth manifold $M$ (possibly with boundary) and choose a partition $M_1,\cdots,M_k$ of $M$ into closed sets\footnote{That is, $\bigcup_i M_i = M$, and for $i\neq j$, the intersection between the interiors of $M_i$ and $M_j$ are empty.} with smooth boundaries, such that for every $i$, there exists a chart $\varphi_i : M_i\to \R^n$. We endow $\R^n$ with the euclidean distance, which defines a distance on $M$ \emph{via} the charts $\phi_i$ (this distance is not necessarily continuous). From now, we study what happens on a single chart, as what happens on the neighbourhoods of the boundaries of these charts ``counts for nothing'' from the Lebesgue measure viewpoint.

Finally, we suppose that the uniform measures on the grids $E_N = \bigcup_i E_{N,i}$ converge to a smooth measure $\lambda$ on $M$ when $N$ goes to infinity.

We also need that the grids behave locally as the canonical grids on the torus.

For every $i$, we choose a sequence $(\kappa_{N,i})_N$ of positive real numbers such that $\kappa_{N,i}\underset{N\to +\infty}{\longrightarrow} 0$. This defines a sequence $E_{N,i}$ of grids on the set $M_i$ by $E_{N,i} = \varphi_i^{-1} (\kappa_{N,i}\Z^n)$. Also, the canonical projection $\pi : \R^n\to \Z^n$ (see Definition~\ref{DefDiscrLin}) allows to define the projection $\pi_{N,i}$, defined as the projection on $\kappa_{N,i}\Z^n$ in the coordinates given by $\varphi_i$:
\[\begin{array}{rcl}
\pi_{N,i} : & M_i & \longrightarrow E_{N,i}\\
            & x   & \longmapsto \varphi_i^{-1}\Big(\kappa_{N,i}\pi\big(\kappa_{N,i}^{-1}\varphi_i(x)\big)\Big).
\end{array}\]

We easily check that under these conditions, Theorem~\ref{TheoA} is still true, that is if we replace the torus $\T^n$ by $M$, the uniform grids by the grids $E_N$, the canonical projections by the projections $\pi_{N,i}$, and Lebesgue measure by the measure $\lambda$.

\section{Proof of Lemma~\ref{ConjPrincip}}\label{ZELINEAR}

Let us summarize the different notations we will use throughout this section. We will denote by $0^k$\index{$0^k$} the origin of the space $\R^k$, and $W^k = ]-1/2,1/2]^{nk}$ (unless otherwise stated). In this section, we will denote $B_R = B_\infty(0,R)$ and $D_c(E)$\index{$D_c$} the density of a ``continuous'' set $E\subset \R^n$, defined as (when the limit exists)
\[D_c(E) = \lim_{R\to+\infty} \frac{\Leb(B_R\cap E)}{\Leb(B_R)},\]
while for a discrete set $E\subset \R^n$, the notation $D_d(E)$\index{$D_d$} will indicate the discrete density of $E$, defined as (when the limit exists)
\[D_d(E) = \lim_{R\to+\infty} \frac{\card(B_R\cap E)}{\card(B_R\cap \Z^n)},\]

We will consider $(A_k)_{k\ge 1}$ a sequence of matrices of $SL_n(\R)$, and denote
\[\Gamma_k = (\widehat{A_k}\circ\dots\circ\widehat{A_1}) (\Z^n).\]
Also, $\Lambda_k$ will be the lattice $M_{A_1,\cdots,A_k} \Z^{n(k+1)}$, with
\begin{equation}\label{DefMat}
M_{A_1,\cdots,A_k} = \left(\begin{array}{ccccc}
A_1 & -\Id &        &        & \\
    & A_2  & -\Id   &        & \\
    &      & \ddots & \ddots & \\
    &      &        & A_k    & -\Id\\
    &      &        &        & \Id
\end{array}\right)\in M_{n(k+1)}(\R),
\end{equation}
and $\widetilde \Lambda_k$ will be the lattice $\widetilde M_{A_1,\cdots,A_k} \Z^{nk}$, with
\begin{equation*}
\widetilde M_{A_1,\cdots,A_k} = \left(\begin{array}{ccccc}
A_1 & -\Id &        &         & \\
    & A_2  & -\Id   &         & \\
    &      & \ddots & \ddots  & \\
    &      &        & A_{k-1} & -\Id\\
    &     &        &          & A_k
\end{array}\right)\in M_{nk}(\R).
\end{equation*}
Finally, we will denote
\[\overline\tau^k(A_1,\cdots,A_k) = D_c\left( W^{k+1} + \Lambda_k \right)\]
the \emph{mean rate of injectivity in time $k$} of $A_1,\cdots,A_k$.

\subsection[A geometric viewpoint to compute the rate of injectivity]{A geometric viewpoint to compute the rate of injectivity in arbitrary times}

We begin by motivating the introduction of model sets by giving an alternative construction of the image sets $(\widehat{A_k}\circ\dots\circ\widehat{A_1}) (\Z^n)$ using this formalism.

Let $A_1,\cdots,A_k\in GL_n(\R)$, then\index{$\Gamma_k$}
\begin{align}
\Gamma_k & = (\widehat{A_k}\circ\dots\circ\widehat{A_1}) (\Z^n)\nonumber\\
         & = \big\{p_2(\lambda)\mid \lambda_k\in \Lambda_k,\, p_1(\lambda)\in W^k\big\}\nonumber\\
				 & = p_2\Big(\Lambda \cap \big(p_1^{-1}(W^k)\big)\Big),\label{CalcGamma}
\end{align}
with $p_1$ the projection on the $nk$ first coordinates and $p_2$ the projection on the $n$ last coordinates. This allows us to see the set $\Gamma_k$ as a model set.

Here, we suppose that the set $p_1(\Lambda_k)$ is dense (thus, equidistributed) in the image set $\im p_1$ (note that this condition is generic among the sequences of invertible linear maps). In particular, the set $\{p_2(\gamma)\mid \gamma\in\Lambda_k\}$ is equidistributed in the window $W^k$.

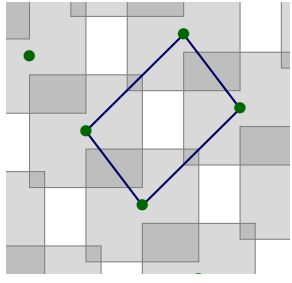
\begin{figure}[t]
\begin{center}
\begin{tikzpicture}[scale=1.5]
\clip (-1.2,-.6) rectangle (1.3,1.8);
\foreach\i in {-1,...,2}{
\foreach\j in {-1,...,3}{
\fill[color=gray,opacity = .3] (.866*\i-.5*\j-.5,.859*\i+.659*\j-.5) rectangle (.866*\i-.5*\j+.5,.859*\i+.659*\j+.5);
\draw[color=gray] (.866*\i-.5*\j-.5,.859*\i+.659*\j-.5) rectangle (.866*\i-.5*\j+.5,.859*\i+.659*\j+.5);
}}
\draw[color=blue!40!black,thick] (0,0) -- (.866,.859) -- (.366,1.518) -- (-.5,.659) -- cycle;
\foreach\i in {-2,...,2}{
\foreach\j in {-2,...,2}{
\draw[color=green!40!black] (.866*\i-.5*\j,.859*\i+.659*\j) node {$\bullet$};
}}
\end{tikzpicture}
\caption{Geometric construction to compute the rate of injectivity: the green points are the elements of $\Lambda$, the blue parallelogram is a fundamental domain of $\Lambda$ and the grey squares are centred on the points of $\Lambda$ and have radii $1/2$. The rate of injectivity is equal to the area of the intersection between the union of the grey squares and the blue parallelogram.}\label{tourp}
\end{center}
\end{figure}

The following property makes the link between the density of $\Gamma_k$ --- that is, the rate of injectivity of $A_1,\cdots,A_k$ --- and the density of the union of unit cubes centred on the points of the lattice $\Lambda_k$ (see Figure~\ref{tourp}). This formula seems to be very specific to the model sets defined by the matrix $M_{A_1,\cdots,A_k}$ and the window $W^k$, it is unlikely that it can be generalized to other model sets.

\begin{prop}\label{CalculTauxModel}
For a generic sequence of matrices $(A_k)_k$ of $SL_n(\R)$, we have
\[D_d(\Gamma_k) = D_c\left( W^k + \widetilde\Lambda_k \right) = \overline\tau^k(A_1,\cdots,A_k).\]
\end{prop}

\begin{rem}
The density on the left of the equality is the density of a discrete set (that is, with respect to counting measure), whereas the density on the right of the equality is that of a continuous set (that is, with respect to Lebesgue measure). The two notions coincide when we consider discrete sets as sums of Dirac masses.
\end{rem}

\begin{rem}\label{conttaukk}
Proposition~\ref{CalculTauxModel} asserts that for a generic sequence of matrices, the rate of injectivity $\tau^k$ in time $k$ coincides with the mean rate of injectivity $\overline\tau^k$, which is continuous and piecewise polynomial of degree $\le nk$ in the coefficients of the matrix.
\end{rem}

%\begin{rem}
%The formula of Proposition~\ref{CalculTauxModel} could be used to compute numerically the mean rate of injectivity in time $k$ of a sequence of matrices: it is much faster to compute the volume of a finite number of intersections of cubes (in fact, a small number) than to compute the cardinalities of the images of a big set $[-R,R]^n \cap \Z^n$.
%\end{rem}

\begin{proof}[Proof of Proposition \ref{CalculTauxModel}]
We want to determine the density of $\Gamma_k$. By Equation~\eqref{CalcGamma}, we have
\[x\in\Gamma_k \iff x\in\Z^n\ \text{and}\ \exists \lambda\in \Lambda_k : x=p_2(\lambda),\, p_1(\lambda) \in W^k.\]
But if $x=p_2(\lambda)$, then we can write $\lambda=(\widetilde\lambda,0^n) + (0^{(k-1)n},-x,x)$ with $\widetilde\lambda\in \widetilde\Lambda_k$. Thus,
\begin{align*}
x\in\Gamma_k & \iff x\in\Z^n\ \text{and}\ \exists \widetilde\lambda\in \widetilde\Lambda_k : (0^{(k-1)n},-x)-\widetilde\lambda	\in W^k\\
             & \iff x\in\Z^n\ \text{and}\ (0^{(k-1)n},x)\in \bigcup_{\widetilde\lambda\in\widetilde\Lambda_k} \widetilde\lambda - W^k.
\end{align*}
Thus, $x\in\Gamma_k$ if and only if the projection of $(0^{(k-1)n},x)$ on $\R^{nk}/\widetilde\Lambda_k$ belongs to $\bigcup_{\widetilde\lambda\in\widetilde\Lambda_k} \widetilde\lambda - W^k$. Then, the proposition follows directly from the fact that the points of the form $(0^{(k-1)n},x)$, with $x\in\Z^n$, are equidistributed in $\R^{nk} / \widetilde \Lambda_k$.
\bigskip

To prove this equidistribution, we compute the inverse matrix of $\widetilde M_{A_1,\cdots,A_k}$:
\[{\widetilde M_{A_1,\cdots,A_k}}^{-1} = \begin{pmatrix}
A_1^{-1} & A_1^{-1}A_2^{-1} &  A_1^{-1}A_2^{-1}A_3^{-1} & \cdots & A_1^{-1}\cdots A_k^{-1}\\
    & A_2^{-1}  & A_2^{-1}A_3^{-1} & \cdots & A_2^{-1}\cdots A_k^{-1}\\
    &      & \ddots & \ddots  & \vdots\\
    &      &        & A_{k-1}^{-1} & A_{k-1}^{-1}A_k^{-1}\\
    &     &        &          & A_k^{-1}
\end{pmatrix}.\]
Thus, the set of points of the form $(0^{(k-1)n},x)$ in $\R^{nk} / \widetilde \Lambda_k$ corresponds to the image of the action
\[\Z^n \ni x \longmapsto
\begin{pmatrix}
A_1^{-1}\cdots A_k^{-1}\\
A_2^{-1}\cdots A_k^{-1}\\
\vdots\\
A_{k-1}^{-1}A_k^{-1}\\
A_k^{-1}
\end{pmatrix}x\]
of $\Z^n$ on the canonical torus $\R^{nk}/\Z^{nk}$. But this action is ergodic (even in restriction to the first coordinate) when the sequence of matrices is generic.
\end{proof}

\subsection[Generically, the asymptotic rate is zero]{Proof of Lemma~\ref{ConjPrincip}: generically, the asymptotic rate is zero}

We now come to the proof of Lemma~\ref{ConjPrincip}. We will use an induction to decrease the rate step by step. Recall that $\overline\tau^k(A_1,\cdots,A_k)$ indicates the density of the set $W^{k+1} + \Lambda_k$.
%
%Unfortunately, if the density of $W^k + \widetilde\Lambda_k$ --- which is generically equal to the density of the $k$-th image $\big(\widehat A_k\circ\cdots\circ \widehat A_1\big)(\Z^n)$ --- is smaller than $1/2$, then we can not apply exactly the strategy of proof of the previous section (see Figure~\ref{FigTauTiers}). For example, if we take
%\[A_1 = \operatorname{diag}(100,1/100)  \quad  \text{and}\quad A_2 = \operatorname{diag}(1/10,10),\]
%then $\big(\widehat{A_2}\circ \widehat{A_1}\big)(\Z^2) = (2\Z)^2$, and for every $B_3$ close to the identity, we have $\tau^3(A_1,A_2,B_3) = \tau^2(A_1,A_2) = 1/100$.
%
%Moreover, if we set $A_k=\operatorname{Id}$ for every $k\ge 2$, then we can set $B_1=A_1$, $B_2=A_2$ and for each $k\ge 2$ perturb each matrix $A_k$ into the matrix $B_k = \operatorname{diag}(1+\delta,1/(1+\delta))$, with $\delta>0$ small. In this case, there exists a time $k_0$ (minimal) such that $\tau^{k_0}(B_1,\cdots,B_{k_0}) < 1/100$. But if instead of $A_{k_0} = \operatorname{Id}$, we have $A_{k_0} = \operatorname{diag}(1/5,5)$, this construction does not work anymore: we should have set $B_k = \operatorname{diag}(1/(1+\delta),1+\delta)$. This suggests that we should take into account the next terms of the sequence $(A_k)_k$ to perform the perturbations. And things seems even more complicated when the matrices are no longer diagonal\dots

More precisely, we will prove that for a generic sequence $(A_k)_{k\ge 1}$, if $\overline\tau^k(A_1,\cdots,A_k)> 1/\ell$, then $\overline\tau^{k+\ell-1}(A_1,\cdots,A_{k+\ell-1})$ is strictly smaller than $\overline\tau^k(A_1,\cdots,A_k)$. More precisely, we consider the maximal number of disjoint translates of $W^k + \widetilde\Lambda_k$ in $\R^{nk}$: we easily see that if the density of $W^k + \widetilde\Lambda_k$ is bigger than $1/\ell$, then there can not be more than $\ell$ disjoint translates of $W^k + \widetilde\Lambda_k$ in $\R^{nk}$(Lemma~\ref{EstimTauxN}). At this point, Lemma~\ref{LemDeFin} states that if the sequence of matrices is generic, then either the density of $W^{k+1} + \widetilde\Lambda_{k+1}$ is smaller than that of $W^k + \widetilde\Lambda_k$ (Figure~\ref{FigLemDeFin}), or there can not be more than $\ell-1$ disjoint translates of $W^{k+1} + \widetilde\Lambda_{k+1}$ in $\R^{n(k+1)}$(see Figure~\ref{FigLemDeFin2}). Applying this reasoning (at most) $\ell-1$ times, we obtain that the density of $W^{k+\ell-1} + \widetilde\Lambda_{k+\ell-1}$ is smaller than that of $W^k + \widetilde\Lambda_k$. For example if $D_c\big(W^k + \widetilde\Lambda_k\big) > 1/3$, then $D_c\big(W^{k+2} + \widetilde\Lambda_{k+2}\big) < D\big(W^k + \widetilde\Lambda_k\big)$ (see Figure~\ref{FigInterCubes3DSupDemi}). To apply this strategy in practice, we have to obtain quantitative estimates about the loss of density we get between times $k$ and $k+\ell-1$.

Remark that with this strategy, we do not need to make ``clever'' perturbations of the matrices: provided that the coefficients of the matrices are rationally independent, the perturbation of each matrix is made independently from that of the others. However, this reasoning does not tell when exactly the rate of injectivity decreases (likely, in most of cases, the speed of decreasing of the rate of injectivity is much faster than the one obtained by this method), and does not say either where exactly the loss of injectivity occurs in the image sets.

We will indeed prove a more precise statement of Lemma~\ref{ConjPrincip}.

{\renewcommand{\thelemme}{\ref{ConjPrincip}}
\begin{lemme}
For a generic sequence of matrices $(A_k)_{k\ge 1}$ of $\ell^\infty(SL_n(\R))$, for every $\ell\in\N$, there exists $\lambda_\ell\in]0,1[$ such that for every $k\in\N$,
\begin{equation}\label{EstimTauxExp}
\tau^{\ell k}(A_1,\cdots,A_{\ell k}) \le \lambda_\ell^k + \frac{1}{\ell}.
\end{equation}
Thus, the asymptotic rate of injectivity $\tau^\infty\big( (A_k)_{k\ge 1}\big)$ is equal to zero.
\end{lemme}
\addtocounter{lemme}{-1}}

The following lemma expresses that if the density of $W^k + \widetilde\Lambda_k$ is bigger than $1/\ell$, then there can not be more than $\ell$ disjoint translates of $W^k + \widetilde\Lambda_k$, and gives an estimation on the size of these intersections.

\begin{lemme}\label{EstimTauxN}
Let $W^k = ]-1/2,1/2]^k$ and $\Lambda\subset \R^k$ be a lattice with covolume 1 such that $D_c(W^k + \Lambda)\ge 1/\ell$. Then, for every collection $v_1,\cdots,v_\ell\in\R^k$, there exists $i\neq i'\in \llbracket 1,\ell\rrbracket$ such that
\[D_c\big((W^k + \Lambda + v_i)\cap (W^k + \Lambda + v_{i'})\big) \ge 2\frac{\ell D_c(W^k + \Lambda)-1}{\ell(\ell-1)}.\]
\end{lemme}

\begin{proof}[Proof of Lemma~\ref{EstimTauxN}]
For every $v\in\R^k$, the density $D_c(W^k + \Lambda+v)$ is equal to the volume of the projection of $W^k$ on the quotient space $\R^k/\Lambda$. As this volume is greater than $1/\ell$, and as the covolume of $\Lambda$ is 1, the projections of the $W^k + v_i$ overlap, and the volume of the points belonging to at least two different sets is bigger than $\ell D_c(W^k + \Lambda)-1$. As there are $\ell(\ell-1)/2$ possibilities of intersection, there exists $i\neq i'$ such that the volume of the intersection between the projections of $W^k+v_i$ and $W^k+v_{i'}$ is bigger than $2(\ell D_c(W^k + \Lambda)-1)/(\ell(\ell-1))$. Returning to the whole space $\R^k$, we get the conclusion of the lemma.
\end{proof}

We will also need the following lemma, whose proof consists in a simple counting argument.

\begin{lemme}\label{DoubleComptage}
Let $\Lambda_1$ be a subgroup of $\R^m$, $\Lambda_2$ be such that $\Lambda_1 \oplus \Lambda_2$ is a lattice of covolume 1 of $\R^m$, and $C$ be a compact subset of $\R^m$. Let $C_1$ be the projection of $C$ on the quotient $\R^m/\Lambda_1$, and $C_2$ the projection of $C$ on the quotient $\R^m/(\Lambda_1\oplus\Lambda_2)$. We denote by
\[a_i = \Leb\big\{x\in C_1 \mid \card\{\lambda_2\in\Lambda_2 \mid x\in C_1+\lambda_2\} = i \big\}\]
(in particular, $\sum_{i\ge 1} a_i = \Leb(C_1)$). Then,
\[\Leb(C_2) = \sum_{i\ge 1} \frac{a_i}{i}.\]
\end{lemme}

In particular, the area of $C_2$ (the projection on the quotient by $\Lambda_1 \oplus \Lambda_2$) is smaller than (or equal to) that of $C_1$ (the projection on the quotient by $\Lambda_1$). The loss of area is given by the following corollary.

\begin{coro}\label{CoroSansNom}
With the same notations as for Lemma~\ref{DoubleComptage}, if we denote by
\[D_1 = \Leb\big\{x\in C_1 \mid \card\{\lambda_2\in\Lambda_2 \mid x\in C_1+\lambda_2\} \ge 2 \big\},\]
then,
\[\Leb(C_2) \le \Leb(C_1) - \frac{D_1}{2}.\]
\end{coro}

Recall that we denote $\widetilde\Lambda_k$ the lattice spanned by the matrix
\[\widetilde M_{A_1,\cdots,A_k} = \left(\begin{array}{ccccc}
A_1 & -\Id &        &         & \\
    & A_2  & -\Id   &         & \\
    &      & \ddots & \ddots  & \\
    &      &        & A_{k-1} & -\Id\\
    &     &        &          & A_k
\end{array}\right)\in M_{nk}(\R),
\]
and $W^k$ the cube $]-1/2,1/2]^{nk}$. The proof of Lemma~\ref{ConjPrincip} will reduce to the following technical lemma.

\begin{lemme}\label{LemDeFin}
For every $\delta>0$ and every $M>0$, there exists $\varep>0$ and an open set of matrices $\mathcal O \subset SL_n(\R)$, which is $\delta$-dense in the set of matrices of norm $\le M$, such that if $\ell\ge 2$ and $D_0>0$ are such that for every collection of vectors $v_1,\cdots,v_{\ell} \in\R^n$, there exists $j,j'\in\llbracket 1,\ell\rrbracket$ such that
\[D_c\bigg(\Big(W^k + \widetilde\Lambda_k + (0^{(k-1)n} , v_j) \Big) \cap \Big(W^k + \widetilde\Lambda_k + (0^{(k-1)n} , v_{j'}) \Big)\bigg)\ge D_0,\]
then for every $B \in \mathcal O$, if we denote by $\widetilde\Lambda_{k+1}$ the lattice spanned by the matrix $\widetilde M_{A_1,\cdots,A_k,B}$,
\begin{enumerate}[(1)]
\item either $D_c(W^{k+1} + \widetilde\Lambda_{k+1}) \le D_c(W^{k} + \widetilde\Lambda_{k}) - \varep D_0/(4\ell)$;
\item or for every collection of vectors $w_1,\cdots,w_{\ell-1} \in\R^n$, there exists $i\neq i'\in\llbracket 1,\ell-1\rrbracket$ such that
\[D_c\bigg(\Big(W^{k+1} + \widetilde\Lambda_{k+1} + (0^{kn} , w_i) \Big) \cap \Big(W^{k+1} + \widetilde\Lambda_{k+1} + (0^{kn} , w_{i'}) \Big)\bigg)\ge \varep D_0 /\ell^2.\]
\end{enumerate}
\end{lemme}

In a certain sense, the conclusion (1) corresponds to an hyperbolic case, and the conclusion (2) expresses that there is a diffusion between times $k$ and $k+1$.

\begin{proof}[Proof of Lemma \ref{LemDeFin}]
Let $\mathcal O_\varep$ be the set of the matrices $B\in SL_n(\R)$ satisfying: for any collection of vectors $w_1,\cdots,w_{\ell-1} \in\R^n$, there exists a set $U\subset\R^n/B\Z^n$ of measure $>\varep$ such that every point of $U$ belongs to at least $\ell$ different cubes of the collection $(Bv + w_i + W^1)_{v\in\Z^n,\,1\le i\le\ell-1}$. In other words, every $x\in\R^n$ whose projection $\overline x$ on $\R^n/B\Z^n$ belongs to $U$ satisfies
\begin{equation}\label{EqLemDeFin}
\sum_{i=1}^{\ell-1} \sum_{v\in \Z^n}  \1_{x\in Bv + w_i + W^1} \ge \ell.
\end{equation}
We easily see that the sets $\mathcal O_\varep$ are open and that the union of these sets over $\varep>0$ is dense (it contains the set of matrices $B$ whose entries are all irrational). Thus, if we are given $\delta>0$ and $M>0$, there exists $\varep>0$ such that $\mathcal O = \mathcal O_\varep$ is $\delta$-dense in the set of matrices of $SL_n(\R)$ whose norm is smaller than $M$.

We then choose $B\in \mathcal O$ and a collection of vectors $w_1,\cdots,w_{\ell-1} \in\R^n$. Let $x\in\R^n$ be such that $\overline x\in U$. By hypothesis on the matrix $B$, $x$ satisfies Equation~\eqref{EqLemDeFin}, so there exists $\ell+1$ integer vectors $v_1,\cdots,v_{\ell}$ and $\ell$ indices $i_1,\cdots,i_{\ell}$ such that the couples $(v_j,i_j)$ are pairwise distinct and that
\begin{equation}\label{EqIxIn}
\forall j\in\llbracket 1,\ell\rrbracket,\quad x\in Bv_j + w_{i_j} + W^1.
\end{equation}
\bigskip

\begin{figure}
\begin{minipage}[t]{.48\linewidth}
\begin{center}
\begin{tikzpicture}[scale=.95]
\draw[color=gray] (-.8,3) -- (5,3);
\draw[very thick,blue, |-|] (-.5,3) -- (.5,3);
\draw[very thick,blue, |-|] (-.5+2.8,3) -- (.5+2.8,3);

\draw[dashed] (-.8,.4) -- (5,.4);
\draw(-.8,.4) node[left]{$x$};
\draw[dashed] (2.4,-1.5) -- (2.4,1.7);
\draw (2.4,-1.5) node[below]{$y$};

\clip (-.6,-1.3) rectangle (4.8,1.5);
\foreach\j in {0,...,1}{
\draw[fill=blue,opacity=.25] (-.5+2.8*\j,-.5) rectangle (.5+2.8*\j,.5);
}
\foreach\i in {-3,...,5}{
\foreach\j in {-1,...,3}{
\draw[fill=blue,opacity=.20] (-.5+\i-1+2.8*\j,-.5+.4*\i-.4) rectangle (.5+\i-1+2.8*\j,.5+.4*\i-.4);
\draw (-.5+\i-1+2.8*\j,-.5+.4*\i-.4) rectangle (.5+\i-1+2.8*\j,.5+.4*\i-.4);
}}
\end{tikzpicture}
\caption[First case of Lemma \ref{LemDeFin}]{First case of Lemma \ref{LemDeFin}, in the case $\ell=3$: the set $W^{k+1} + \widetilde\Lambda_{k+1}$ auto-intersects.}\label{FigLemDeFin}
\end{center}
\end{minipage}\hfill
\begin{minipage}[t]{.48\linewidth}
\begin{center}
\begin{tikzpicture}[scale=.95]
\draw[color=gray] (-.8,3) -- (5,3);
\draw[very thick,blue, |-|] (-.5,3) -- (.5,3);
\draw[very thick,blue, |-|] (-.5+2.8,3) -- (.5+2.8,3);

\draw[dashed] (-.8,.4) -- (5,.4);
\draw(-.8,.4) node[left]{$x$};
\draw[dashed] (1.4,-1.5) -- (1.4,1.7);
\draw (1.4,-1.5) node[below]{$y$};

\clip (-.6,-1.3) rectangle (4.8,1.5);
\foreach\j in {0,...,1}{
\draw[fill=blue,opacity=.25] (-.5+2.8*\j,-.5) rectangle (.5+2.8*\j,.5);
}
\foreach\i in {-3,...,2}{
\foreach\j in {-1,...,2}{
\draw[fill=blue,opacity=.20] (-.5+\i+2.8*\j,-.5+.8*\i) rectangle (.5+\i+2.8*\j,.5+.8*\i);
\draw (-.5+\i+2.8*\j,-.5+.8*\i) rectangle (.5+\i+2.8*\j,.5+.8*\i);
\draw[fill=gray,opacity=.15] (-.5+\i+2.8*\j,-.5+.8*\i+1.2) rectangle (.5+\i+2.8*\j,.5+.8*\i+1.2);
\draw (-.5+\i+2.8*\j,-.5+.8*\i+1.2) rectangle (.5+\i+2.8*\j,.5+.8*\i+1.2);
}}
\end{tikzpicture}
\caption[Second case of Lemma \ref{LemDeFin}]{Second case of Lemma \ref{LemDeFin}, in the case $\ell=3$: two distinct vertical translates of $W^{k+1} + \widetilde\Lambda_{k+1}$ intersect (the first translate contains the dark blue thickening of $W^k + \widetilde\Lambda_k$, the second is represented in grey).}\label{FigLemDeFin2}
\end{center}
\end{minipage}
\end{figure}

The following formula makes the link between what happens in the $n$ last and in the $n$ penultimates coordinates of $\R^{n(k+1)}$:
\begin{equation}\label{EqChangeDim}
W^{k+1} + \widetilde\Lambda_{k+1} + \big(0^{(k-1)n},0^n,w_{i_j}\big) = W^{k+1} + \widetilde\Lambda_{k+1} + \big(0^{(k-1)n},-v_j,w_{i_j} + Bv_j\big),
\end{equation}
(we add a vector belonging to $\widetilde\Lambda_{k+1}$). 

We now apply the hypothesis of the lemma to the vectors $-v_1,\cdots,-v_{\ell+1}$: there exists $j\neq j'\in\llbracket 1,\ell\rrbracket$ such that
\begin{equation}\label{EqInter}
D_c\bigg(\Big(W^k + \widetilde\Lambda_k + (0^{(k-1)n} , -v_j)\Big) \cap \Big(W^k + \widetilde\Lambda_k + (0^{(k-1)n} , -v_{j'})\Big)\bigg) \ge D_0.
\end{equation}
Let $y$ be a point belonging to this intersection. Applying Equations~\eqref{EqIxIn} and \eqref{EqInter}, we get	that
\begin{equation}\label{EqLemDeFinCentr}
(y,x) \in W^{k+1} + \big(\widetilde\Lambda_k,0^n\big) + \big(0^{(k-1)n}, -v_j, w_{i_j} + Bv_j\big)
\end{equation}
and the same for $j'$.

Two different cases can occur.
\begin{enumerate}[(i)]
\item Either $i_j = i_{j'}$ (that is, the translation vectors $w_{i_j}$ and $w_{i_{j'}}$ are equal). As a consequence, applying Equation~\eqref{EqLemDeFinCentr}, we have
\begin{align*}
(y,x) + \big(0^{(k-1)n}, v_j, -Bv_j-w_{i_j}\big) \in & \Big( W^{k+1} + \big(\widetilde\Lambda_k,0^n\big) \Big)\cap\\
                                                     & \Big( W^{k+1} + \big(\widetilde\Lambda_k,0^n\big) + v'\Big),
\end{align*}
with
\[v' = \big(0^{(k-1)n}, -(v_{j'}-v_j), B(v_{j'}-v_j)\big) \in \widetilde\Lambda_{k+1} \setminus\widetilde\Lambda_k.\]
This implies that the set $W^{k+1} + \widetilde\Lambda_{k+1}$ auto-intersects (see Figure~\ref{FigLemDeFin}).
\item Or $i_j \neq i_{j'}$ (that is, $w_{i_j}\neq w_{i_{j'}}$). Combining Equations~\eqref{EqLemDeFinCentr} and \eqref{EqChangeDim} (note that $\big(\widetilde\Lambda_k,0^n\big) \subset \widetilde\Lambda_{k+1}$), we get
%But we have Equation~\eqref{EqChangeDim}. Thus, as these two sets intersect in both $kn$ firsts and $n$ last coordinates, we have 
\[(y,x)\in \Big(W^{k+1} + \widetilde\Lambda_{k+1} + \big(0^{kn}, w_{i_j}\big)\Big) \cap \Big(W^{k+1} + \widetilde\Lambda_{k+1} + \big(0^{kn},w_{i_{j'}}\big)\Big).\]
This implies that two distinct vertical translates of $W^{k+1} + \widetilde\Lambda_{k+1}$ intersect (see Figure~\ref{FigLemDeFin2}).
\end{enumerate}

We now look at the global behaviour of all the $x$ such that $\overline x\in U$. Again, we have two cases.
\begin{enumerate}[(1)]
\item Either for more than the half of such $\overline x$ (for Lebesgue measure), we are in the case (i). To each of such $\overline x$ corresponds a translation vector $w_i$. We choose $w_i$ such that the set of corresponding $\overline x$ has the biggest measure; this measure is bigger than $\varep/\big(2(\ell-1)\big)\ge \varep/(2\ell)$. Using the notations of Corollary~\ref{CoroSansNom}, we get that the density $D_1$ of the auto-intersection of $W^{k+1} + \widetilde\Lambda_{k+1} + (0,w_i)$ is bigger than $D_0\varep/(2\ell)$. This leads to (using  Corollary~\ref{CoroSansNom})
\[D_c(W^{k+1} + \widetilde\Lambda_{k+1}) < D_c(W^{k} + \widetilde\Lambda_{k}) - \frac{D_0 \varep}{4\ell}.\]
In this case, we get the conclusion (1) of the lemma.
\item  Or for more than the half of such $\overline x$, we are in the case (ii). Choosing the couple $(w_i,w_{i'})$ such that the measure of the set of corresponding $\overline x$ is the greatest, we get 
\[D_c\bigg(\Big(W^{k+1} + \widetilde\Lambda_{k+1} + (0^{kn} , w_i) \Big) \cap \Big(W^{k+1} + \widetilde\Lambda_{k+1} + (0^{kn} , w_{i'}) \Big)\bigg)\ge \frac{D_0 \varep}{(\ell-1)(\ell-2)}.\]
In this case, we get the conclusion (2) of the lemma.
\end{enumerate}
\end{proof}

\begin{figure}[t]
\begin{center}
\includegraphics[width=.8\linewidth]{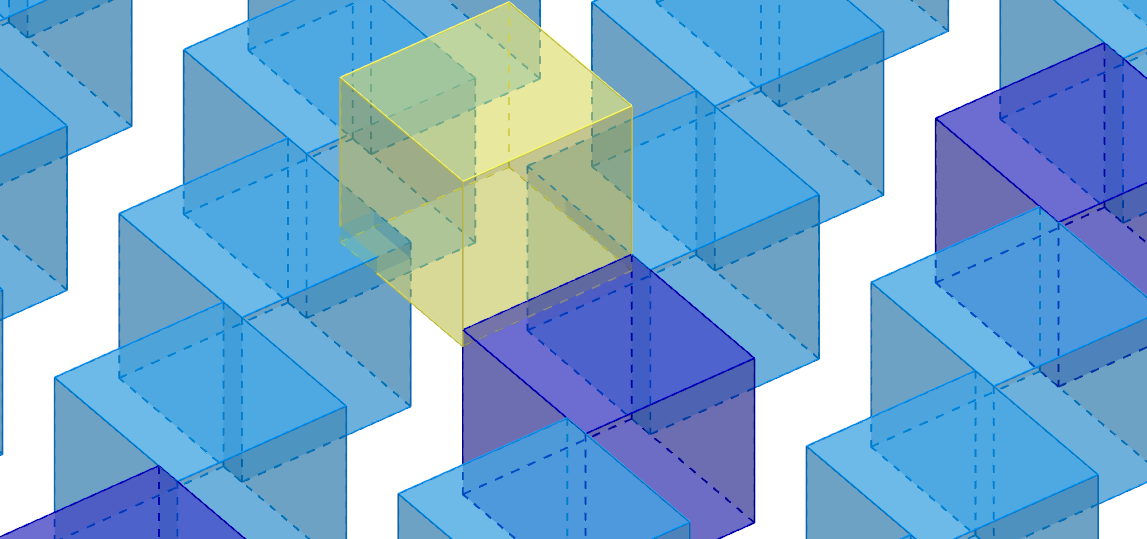}
\caption[Intersection of cubes, rate bigger than $1/3$]{Intersection of cubes in the case where the rate is bigger than $1/3$. The thickening of the cubes of $W^k + \widetilde\Lambda_k$ is represented in dark blue and the thickening of the rest of the cubes of $W^{k+1} + \widetilde\Lambda_{k+1}$ is represented in light blue; we have also represented another cube of $W^{k+2} + \widetilde\Lambda_{k+2}$ in yellow. We see that if the projection on the $z$-axis of the centre of the yellow cube is smaller than 1, then there is automatically an intersection between this cube and one of the blue cubes.}\label{FigInterCubes3DSupDemi}
\end{center}
\end{figure}

We can now prove Lemma~\ref{ConjPrincip}.

\begin{proof}[Proof of Lemma~\ref{ConjPrincip}]
We proceed by induction on $k$. Suppose that $\widetilde\Lambda_k$ is such that $D_c(W^k + \widetilde\Lambda_k)>1/\ell$. Then, Lemma~\ref{EstimTauxN} ensures that it is not possible to have $\ell$ disjoint translates of $W^k + \widetilde\Lambda_k$. Applying Lemma~\ref{LemDeFin}, we obtain that either $D_c(W^{k+1} + \widetilde\Lambda_{k+1})<D_c(W^k + \widetilde\Lambda_k)$, or 
it is not possible to have $\ell-1$ disjoint translates of $W^{k+1} + \widetilde\Lambda_{k+1}$. And so on, applying Lemma~\ref{LemDeFin} at most $\ell-1$ times, there exists $k'\in \llbracket k+1,k+\ell-1\rrbracket $ such that $W^{k'} + \widetilde\Lambda_{k'}$ has additional auto-intersections. Quantitatively, combining Lemmas~\ref{EstimTauxN} and \ref{LemDeFin}, we get
\[D_c\big(W^{k+\ell-1} + \widetilde\Lambda_{k+\ell-1}\big) \le D\big(W^k + \widetilde\Lambda_k\big) - \frac{\varep}{4\ell} \left(\frac{\varep}{\ell^2}\right)^{\ell-1} 2\frac{\ell D_c(W^k + \widetilde\Lambda_k)-1}{\ell(\ell-1)},\]
thus
\[D_c\big(W^{k+\ell-1} + \widetilde\Lambda_{k+\ell-1}\big) - 1/\ell \le  \left( 1 -\frac12 \left(\frac{\varep}{\ell^2}\right)^\ell\right) \Big(D_c\big(W^k + \widetilde\Lambda_k\big) - 1/\ell \Big),\]
in other words, if we denote $\overline\tau^k = \overline\tau^k(B_1,\cdots,B_k)$ and $\lambda_\ell = 1 - \left(\frac{\varep}{\ell^2}\right)^\ell$,
\begin{equation}\label{EqFinChap}
\overline\tau^{k+\ell-1} - 1/\ell \le  \lambda_\ell\Big( \overline\tau^k - 1/\ell \Big).
\end{equation}
This implies that for every $\ell>0$, the sequence of rates $\overline\tau^k$ is smaller than a sequence converging exponentially fast to $1/\ell$: we get Equation~\eqref{EstimTauxExp}. In particular, the asymptotic rate of injectivity is generically equal to zero.
\end{proof}

\bibliographystyle{amsalpha}
\bibliography{../../Biblio}

\end{document}